\documentclass[11pt]{article}
\usepackage{amsmath, amsfonts, amsthm, amssymb}
\usepackage{graphics}
\usepackage{epsfig}
\usepackage[]{pstricks,pst-node,pst-text,pst-3d}

\usepackage{fullpage}

\newcommand{\trp}{^{\text{T}}}

\newcommand{\B}{\mathcal{B}}

\newcommand{\cof}{\mathrm{cof\, }}
\newcommand{\RR}{\mathbb{R}}
\newcommand{\del}{\partial}
\newcommand{\eps}{\varepsilon}

\newcommand{\f}[2]{\frac{#1}{#2}}
\newcommand{\gr}[1]{\nabla {#1}}
\newcommand{\pd}[2]{\frac{\partial #1}{\partial #2}}
\newcommand{\BBR}[1]{\left( #1 \right)}
\newcommand{\BBS}[1]{\left[ #1 \right]}
\newcommand{\BBF}[1]{\left\{ #1 \right\}}
\newcommand{\BBN}[1]{\left\| #1 \right\|}
\newcommand{\BBA}[1]{\left | #1 \right |}
\newcommand{\txq}[1]{\quad \mbox{#1} \quad}
\newcommand{\txc}[1]{\; \mbox{#1} \;}

\newcommand{\tx}{\mbox}
\newcommand{\mdd}[1]{M^{#1\times #1}}
\newcommand{\mddplus}[1]{M^{#1\times #1}_+}
\newcommand{\diagmtx}[3]{\begin{bmatrix} #1 \; \quad \; \quad\\\quad \; #2 \; \quad \\\quad \; \quad \; #3\end{bmatrix}}

\newtheorem{lmm}{Lemma}
\newtheorem{thm}{Theorem}

\newtheorem{Remark}{Remark}

\title{A variational approximation scheme for radial polyconvex elasticity
that preserves the positivity of Jacobians\\
{\footnotesize (In: {\it Communications in Mathematical Sciences} 10-1, pp. 87-115 (2012))} }

\author{Alexey Miroshnikov\thanks {Department of Mathematics, University of Maryland,
College Park, USA}
\;   and \, Athanasios E. Tzavaras
\thanks{Department of Applied Mathematics,
University of Crete, Heraklion, Greece.}}

\date{}
\begin{document}

\maketitle

\begin{center}
{\it To our colleague, teacher and friend David Levermore \\
on his 60th anniversary}
\end{center}

\begin{center}
\textbf{Abstract}
\end{center}
\noindent We consider the equations describing the dynamics of
radial motions for isotropic elastic materials; these form a
system of non-homogeneous conservation laws. We construct a
variational approximation scheme that decreases the total
mechanical energy and at the same time leads to physically
realizable motions that avoid interpenetration of matter.

\begin{center} Mathematics Subject Classification:
35L70
49J45
74B20
74H20
\end{center}

\section{Introduction}

The equations describing radial motions of nonlinear, isotropic, elastic materials take the form
\begin{equation}\label{RPDEINTRO}
\begin{aligned}
       w_{tt}
        &=\frac{1}{R^2} \pd{}{R}\BBR{R^{2} \pd{\Phi}{v_1} \big ( w_R,\frac{w}{R},\frac{w}{R} \big )}
         -\frac{1}{R}\sum_{i=2}^{3} \pd{\Phi}{v_i} \big (w_R,\frac{w}{R},\frac{w}{R} \big ) \, . \\
\end{aligned}
\end{equation}
Here, $y$ stands for a radial motion $y(x,t) =w(R,t)\f{x}{R}$,  $R=|x|$, $x \in\RR^3$, and
\eqref{RPDEINTRO} monitors the evolution of its amplitude $w(R,t)$.
A necessary condition for $y$ to represent a physically realizable motion is $\det F >0$ with $F = \gr y$.
In the radial case, it dictates
\begin{equation}\label{IMPINTRO}
w_R(w/R)^2>0 \, ,
\end{equation}
and is also a sufficient condition for avoiding interpenetration of matter.

The constitutive properties of hyperelastic materials are completely determined by the
stored energy function $W(F) : \mddplus{3} \to [0,\infty)$, which - due to frame indifference -
has to be invariant under rotations.
For isotropic elastic materials $W(F)=\Phi(v_1,v_2,v_2)$, where $\Phi$ is a symmetric function of
the principal stretches $v_1,v_2,v_3$ of $F$, see \cite{Tr}.
Convexity of the stored energy is, in general,  incompatible with certain physical requirements
and  is not a natural assumption. For instance, in order to avoid
interpenetration of matter the stored energy should increase without bound  as $\det F\to 0^+ $
so that compression of a finite volume down to a point would cost infinite energy.
This behavior is inconsistent with simultaneously requiring convexity and
invariance of the stored energy under rotations.
As an alternative, the assumption of polyconvexity \cite{ball77}
is often employed, which postulates that
$$
W (F) = \sigma (F, \cof F,\det F)
$$
with $\sigma$ a convex function of the null-Lagrangian vector $(F,\cof F,\det F)$,
and encompasses certain physically realistic models ({\it e.g.} \cite[Sec 4.9, 4.10]{Ciarlet}).
In this work, we employ a specific form of polyconvex stored energy,
\begin{equation}\label{STENERGYSPEC}
\begin{aligned}
&W(F)  =  \Phi(v_1,v_2,v_3)
\\
&= \phi(v_1) + \phi(v_2) + \phi(v_3)
          +g(v_2v_3) + g(v_1 v_3) + g (v_1 v_2) + h(v_1 v_2 v_3) \, ,
\end{aligned}
\end{equation}
where $\phi$, $g$ and $h$ are convex functions and
$h(\delta) \to +\infty$ as $\delta\to 0+$.\par\medskip

Equation \eqref{RPDEINTRO} may be recast as a system of inhomogeneous balance laws,
\begin{equation}\label{SYSTEMINTRO}
\begin{aligned}
       v_{t}
        &=\frac{1}{R^2} \pd{}{R}\BBR{R^{2} \pd{\Phi}{v_1} \big ( u,\frac{w}{R},\frac{w}{R} \big )}
         -\frac{1}{R}\sum_{i=2}^{3} \pd{\Phi}{v_i} \big (u,\frac{w}{R},\frac{w}{R} \big ) \, .
         \\
       u_{t} &= v_{R}
       \\
       w_t &= v \, ,
\end{aligned}
\end{equation}
where $u = w_R$, and $v=w_t$. The system admits the entropy-entropy flux pair
\begin{equation}\label{ENERGYINTRO}
R^2 \partial_t \BBR{ \f{v^2}{2} + \Phi \big ( u,\frac{w}{R},\frac{w}{R} \big ) }
- \partial_{R}  \BBR{ R^2 \, v \,  \frac{\del \Phi}{\del v_1} \big ( u,\frac{w}{R},\frac{w}{R} \big ) }
= 0 \, ,
\end{equation}
which expresses the conservation  of mechanical energy along smooth solutions. For polyconvex stored energies, the "entropy"
$$
\eta = \frac{1}{2} v^2 + \Phi \big ( u,\frac{w}{R},\frac{w}{R} \big )
$$
is not convex, what causes various difficulties in applying the
general theory of conservation laws. Nevertheless, for
three-dimensional elastodynamics, there are available nonlinear
transport identities for the null-Lagrangians \cite{Qn}, which
allow to view the equations of elasticity as constrained evolution
of an enlarged symmetrizable system \cite{Tz2, Dm} equipped with a
relative entropy identity \cite{LT}. The enlarged system suggests
a variational approximation scheme for polyconvex
elasticity that dissipates the mechanical energy \cite{Tz2},  and
which, in the one-dimensional case, produces entropy weak solutions
\cite{Tz}. Conceptually similar structures  are available in
models of electromagnetism leading to augmented symmetrizable
hyperbolic systems \cite{Brenier,Serre,Serre2}.

The above results do not take into account the constraint of positive determinant, necessary
to interpret $y$ as a physically realizable motion. In this article, we consider the equations of
radial elasticity  \eqref{RPDEINTRO} and proceed to devise a variational approximation
scheme that on one hand preserves the positivity of determinants \eqref{IMPINTRO}
and on the other produces a time-discretized variant  of entropy dissipation.
As in \cite{Tz2}, the scheme is based on transport identities for the
null-Lagrangians. Null-Lagrangians are potential energies $\Psi (v_1,v_2,v_3 ; R)$ for which the functional
\begin{equation}\label{ACTIONINTRO}
I[w] =  \int_0^1 \Psi \big ( \big (  w_R,\f{w}{R} , \f{w}{R} \big
) ; R \big ) \, dR
\end{equation}
has variational derivative zero. They satisfy
\begin{equation}\label{NLAGRINTRO}
- \partial_{R} \BBR{\Psi_{,1}} + R^{-1} \BBR{ \Psi_{,2} +
\Psi_{,3}}=0 \quad \text{for all functions} \;  w(R) \, .
\end{equation}
where  $\Psi_{,i} := \frac{\partial \Psi}{\partial v_i}$,
$i = 1,2,3$,  stands for the partial derivative. The null-Lagrangians are
computed to be the functions $v_1$, $v_1 v_2 R$, $v_1 v_3 R$ or $v_1 v_2 v_3 R^2$.
Along solutions of the dynamical problem, each null-Lagrangian satisfies  the transport identity
\begin{equation}\label{NLEVOLINTRO}
\partial_t \Psi = \partial_{R} \BBR{\Psi_{,1}\, v} \, ,
\end{equation}
with $\Psi$ and $\Psi_{,i}$ are evaluated at $\Gamma = \BBR{w_R, w/R, w/R, R}$. The identities
\eqref{NLEVOLINTRO} allow  to embed the system \eqref{SYSTEMINTRO} into the symmetrizable
first-order evolution \eqref{EXTSYSPOS} in Section \ref{extensionone}.


The enlarged system, in the form  \eqref{EXTSYSPOS},  cannot handle
the positivity of determinants constraint.
For this reason we follow an alternative strategy,
combining a change of
variables suggested in Ball \cite{Ba} (for the equilibrium
problem) with the idea of extensions based on null-Lagrangians, and carry out an
alternative extended system. We set $\rho=R^3$, $\alpha=w^3$,
$\beta = w_R/R^2$, $\gamma=w^2$ and let
\begin{equation}\label{XIDEFINTRO}
\Xi=\;\BBR{\beta \rho^{2/3}, \f{\alpha}{\rho},\f{\alpha}{\rho},
\f{\gamma} {\rho^{1/3}}, \f{3\gamma_{\rho}}{2}\rho^{2/3},
\f{3\gamma_{\rho}}{2}\rho^{2/3},\alpha_{\rho}\rho^{2/3}}.
\end{equation}
The second extension has four actual unknowns $v$, $\alpha$,
$\beta$ and $\gamma$, and is the
symmetrizable system listed in \eqref{EXPLSYS} of Section \ref{extensiontwo}
endowed with the entropy pair
\begin{equation}\label{EXTSYSPOSRHOINTRO}
\partial_t \BBR{ \f{v^2}{2} + G(\Xi) } - \partial_{\rho} \BBR{3 \rho^{2/3} \, G_{,i}(\Xi)
\, \Omega_{,1}^i(\Gamma) \,v}= 0 \ ,
\end{equation}
where $G$ is defined in \eqref{GDEF2} and is (assumed) convex and $\Gamma$
is as in \eqref{GAMMADEF2}.

The extended system \eqref{EXPLSYS} is discretized in time using
an implicit-explicit scheme. It is the Euler-Lagrange equations of  the variational
problem: given $v_0$ and $\Xi^0$ defined via $\alpha_0$, $\beta_0$
and $\gamma_0$ as in (\ref{XIDEFINTRO}), minimize
\begin{equation}\label{MININTRO}
I(\alpha,\beta,\gamma,v)=\int_{0}^{1} \, \f{1}{2}(v-v_0)^2 +
G(\Xi) \, d\rho
\end{equation}
over the set of admissible functions
\begin{eqnarray}
\label{ADMDEFINTRO}
\begin{aligned}
\mathcal{A}_{\lambda}= \Big \{ (\alpha,\beta,\gamma,v)\in X:\,
&\alpha(0) \geqslant 0,\,\alpha(1)=\lambda,\,\alpha'
>0 \;\tx{a.e.} \; and
\\
&I(\alpha,\beta,\gamma,v)<\infty,\,\f{\BBR{\beta  - \beta_0 }}{h}=3v',
\\
&\f{\BBR{\alpha - \alpha_0}}{h} = 3 {\alpha_0}^{2/3}v,\,
\f{\BBR{\gamma - \gamma_0}}{h} = 2 {\alpha_0}^{1/3}v \Big \}.
\end{aligned}
\end{eqnarray}
The differential constraints in \eqref{ADMDEFINTRO} are affine,
the condition $\alpha(1)=\lambda$ corresponds to the imposed
boundary condition $y(x)=\lambda x$, $x\in \partial \B$, while
$\alpha'>0$  secures the positivity of determinants (\ref{IMPINTRO}).
We prove the existence and uniqueness of a minimizer for the
functional $I$ over $\mathcal{A}_{\lambda}$ and that the
minimizer is a weak solution to the corresponding Euler-Lagrange
equations, that is, a solution of the time-discrete scheme.
The analysis of the minimization problem \eqref{MININTRO}-\eqref{ADMDEFINTRO}
uses direct methods of the calculus of variations, in the spirit of \cite{Ba}, with
the novel element of accounting for the evolutionary constraints in \eqref{ADMDEFINTRO}.

In continuum physics, weak solutions of a system of conservation
laws are required to satisfy entropy inequalities of the form
\begin{equation}\label{IRREVINEQ}
\partial_t \eta + \partial_{\alpha} q_{\alpha} \leq 0 \, .
\end{equation}
Such inequalities reflect irreversibility and
originate from the second law of thermodynamics. For instance,
admissible shocks of the elasticity equations are required to dissipate the mechanical
energy. Accordingly, approximating schemes are expected to respect such behaviors
and produce entropy dissipating solutions in the limit.
The variational scheme studied here turns out to satisfy
a discrete version of the entropy inequality
\begin{equation}\label{DISENTDISINTRO}
\f{ \BBR{\f{v^2}{2} + G(\Xi)} - \BBR{\f{{v_0}^2}{2} + G(\Xi^0)} }{h}
-\f{d}{d\rho} \BBR{3\rho^{2/3} G_{,i}(\Xi)\Omega^i_{,1}(\Gamma^0)v} \leqslant 0
\end{equation}
(see Section \ref{secvarapp}). In addition, the approximants satisfy $\alpha_{\rho}>0$
the transformed version  of \eqref{IMPINTRO}.
Finally, if the constructed approximants converge pointwise as the
time-step $h \to 0$,  then the limit will satisfy
the mechanical energy dissipation  inequality
\begin{equation}\label{EXTSYSPOSRHOINEQINTRO}
\partial_t \BBR{ \f{v^2}{2} + G(\Xi) } - \partial_{\rho} \BBR{3 \rho^{2/3} \, G_{,i}(\Xi)
\, \Omega_{,1}^i(\Gamma) \,v} \leqslant 0.
\end{equation}

The paper is organized as follows. In Section \ref{secprel} we outline the derivation of
the equations of radial elasticity and list various mechanical considerations relevant to
this work.
Section \ref{secnullag} contains a discussion of null-Lagrangians and the properties of the two
symmetrizable extensions of \eqref{SYSTEMINTRO}  pursued.
Section \ref{secvarapp} introduces the time-discrete scheme and its relation to a variational problem.
In Section \ref{seceximin} we consider the minimization problem \eqref{MININTRO}
and prove Theorems \ref{EXISTENCE} and \ref{UNIQUENESS} regarding existence
and uniqueness of minimizers. The Euler-Lagrange equations associated to the minimization
problem are derived in Theorem \ref{ELWEAK} of Section \ref{seceullag}, and the regularity of minimizers
is discussed in Section \ref{secreg}. The fact that minimizers satisfy the time-discretized
version of the entropy dissipation inequality \eqref{DISENTDISINTRO} is proved in
Section \ref{secvarapp}.

%

\section{Preliminaries }
\label{secprel}

We consider the equations of nonlinear elasticity
\begin{eqnarray}
\label{MPDE}
\left\{
\begin{aligned}
y_{tt} &=\mathrm{div} S(\nabla y)     \qquad  \text{in   $\B \times (0,\infty)$} \\
y(x,t) &=\lambda x, \qquad \qquad \text{  on  $\partial\B \times [0,\infty)$} \\
&\det{\gr{y}} >0, \qquad (x,t) \in \B \times [0,\infty)\\
\end{aligned}
\right.
\end{eqnarray}
on the unit ball $\B= \{x \in \RR^n \colon|x|<1\}$, subject to uniform stretching at the boundary
and initial conditions
\begin{equation}
y(x,0) = y_0(x) \, , \; y_t (x,0) = v_0(x) \, ,  \quad x \in \B \, .
\end{equation}
In order for the geometric mapping  $y:\B \times [0,\infty) \to \RR^n $
to correspond to a physically realizable {\emph{motion}} we have to exclude
interpenetration of matter. As a minimum requirement the condition
$\det{\gr{y}}>0$ is imposed.

Let $\mdd{n}$ be the real ${n\times n}$ matrices,
$\mddplus{n}=\{F \in \mdd{n} \colon \det F >0 \}$, and
let $SO(n)$ denote the set of proper rotations.
The Piola-Kirchhoff stress is a mapping $S:\mddplus{n} \to \mdd{n}$
and for {\emph{hyperelastic}} materials it is defined by the formula
\begin{equation}
S(F)=\partial W(F)/\partial F.
\end{equation}
where  $W:\mddplus{n} \to \RR^n$ is the stored-energy function of the elastic body.

We assume that the stored energy function $W$ satisfies the physical requirement of
{\emph{frame-indifference}} and that the elastic material is {\emph{isotropic}}. Then,
    \begin{equation}\label{FRISP}
        W(QF) =   W(F)  = W(FQ)  \quad \forall F \in \mddplus{n} , \; Q \in SO(n)
    \end{equation}
and (see Truesdell and Noll \cite[pp 28, 317]{Tr}) there exists
a  symmetric function
$$
\Phi : \RR_{+}^n = \{ x \in \RR^n :  x_i > 0 \; \forall i \}  \to \RR
$$
such that
\begin{equation}
      W(F)=\Phi (v_1,\dots,v_n) \quad \quad \forall  F \in \mddplus{n},
\end{equation}
where $v_1,\dots,v_n$ are the singular values of $F$, i.e.
the eigenvalues of $(F\trp F )^{1/2}$. We note
 that the symmetry of $\Phi$ implies
\begin{equation}\label{PHIDERPROP}
\pd{\Phi}{v_i}(a,b,\dots,b)=\pd{\Phi}{v_j}(a,b,\dots,b),\;\;\;i,j\geqslant2,\;
a,b\in \RR_+.
\end{equation}
It is easy to check that for hyperelastic, isotropic materials, frame-indifference implies
\begin{equation}\label{PKP}
        S(QFQ\trp)=QS(F)Q\trp  \, , \quad \mbox{for all} \;  Q \in
        SO(n).
\end{equation}

\subsection{Radial Elasticity}

A function $f:\B \backslash \{0\} \to \RR^n$ is called {\emph{radial}} if
\begin{equation*}
f(x)= w(R)\f{x}{R}, \txc{}  R=|x|,
\end{equation*}
where $w:[0,\infty) \to [0,\infty)$.
The space of {\emph{deformations} of $\B$ is denoted by
\begin{equation*}
    \tx{Def}^{\;p}(\B)=\BBF{f \in W_{1}^{p}\BBR{\B,\RR^n} \txc{:} \tx{det}{\gr{f}} > 0 \txc{a.e.}}.
\end{equation*}

\begin{lmm}[J. Ball \cite{Ba}]\label{radgrad}
Let $f$ be a radial function. 
Then $f \in \tx{Def}^{\;p}(\B)$ if and only if $w$ is absolutely
continuous on $(0,1)$ and satisfies $w_R(w/R)^{n-1}>0$ almost everywhere, and
\begin{equation*}
    \int\limits_{0}^{1}\BBR{\BBA{w'}^p + \BBA{w/R}^p }R^{n-1}dR \: < \:
    \infty.
\end{equation*}
In this case the weak derivatives of $f$ are given
by
\begin{equation*}
    \nabla f = \frac{w}{R}{\mathbf I}+\BBR{w' -
    \frac{w}{R}}\frac{x\otimes x}{R^2} \txq{a.e.} x \in \B .
\end{equation*}
\end{lmm} \par\medskip

Our next goal is to consider the problem \eqref{MPDE} and to recast it for
radial motions
\begin{equation}
\label{radmotion}
        y(x,t)=w(R,t) \frac{x}{R} \quad \mbox{for} \quad x\neq 0,
\end{equation}
where $w: [0,1) \times [0,\infty) \to \RR$ satisfies
$w(R,t)\geqslant 0$.  Lemma \ref{radgrad} implies
\begin{equation}
    \nabla y = \frac{w}{R}{\mathbf I}+\BBR{w_R -
    \frac{w}{R}}\frac{x\otimes x}{R^2} \txq{a.e.} x \in \B
\end{equation} \\
and hence the eigenvalues of $\gr{y}$ are expressed as
\begin{equation*}
    v_1=w_R,\: v_2 = ... = v_n =w/R \, .
\end{equation*}
The requirement
\begin{equation}
\det{\gr{y}}=w_R(w/R)^{n-1}>0
\end{equation}
dictates $w_R,\f{w}{R}>0$.
Since $\nabla y$ is symmetric and positive definite, the singular values of $\gr{y}$
coincide with its eigenvalues, the stored energy takes the form
$$
W(\gr{y}) = \Phi\BBR{w_R,\f{w}{R}, ... , \f{w}{R}} \, ,
$$
and property (\ref{PKP}) implies that the Piola-Kirchhoff
stress can be expressed as (see e.g. J.Ball \cite{Ba})
\begin{equation*}
\begin{split}
        S(\gr{y})&=\Phi_{,2}\BBR{w_R,w/R,\dots,w/R}{\mathbf I}+{} \\
        &\BBS{\Phi_{,1}\BBR{w_R,w/R,\dots,w/R} -\Phi_{,2}\BBR{w_R,w/R,\dots,w/R}} \frac{x \otimes
        x}{R^2} \, ,
\end{split}
\end{equation*}
where $\Phi_{, j} := \frac{\partial \Phi}{\partial v_j}$, $j=1,2,3$.
For radial motions, the system (\ref{MPDE}) then takes the form,
\begin{equation}\label{RPDE}
\begin{aligned}
       R^{n-1} \partial_{tt} w
        &=  \pd{}{R}\BBR{R^{n-1}\Phi_{,1}(w_R,\ldots,w/R)}
         -R^{n-2} \sum_{i=2}^{n}\Phi_{,i}(w_R,\dots,w/R)
         \\
       w(1, t )&=\lambda, \quad  w_R\BBR{w/R}^{n-1} >0, \;\; (R,t) \in (0,1)\times
        [0,\infty),
\end{aligned}
\end{equation}
of a second order equation describing the evolution of $w(R,t)$ subject to the constraint
\eqref{RPDE}$_2$. The latter expresses the requirement that matter cannot interpenetrate unto itself.


\subsection{Polyconvex Stored Energy for $n=3$}

From now on we fix the number of dimensions to $n=3$ and
assume that the stored energy $W:\mddplus{3} \to \RR^3$ is {\bf
\emph{polyconvex}}, that is
\begin{equation*}
W(F)=\Bar{\Bar G}\BBR{F,\cof F,\det F}
\end{equation*} for some \emph{convex}
function $\Bar{\Bar G}:\mddplus{3} \times \mddplus{3} \times
\RR_{+} \to \RR$.

By the polar decomposition theorem any matrix $F \in \mddplus{3}$ is expressed in the form
$F = R U$ with $R \in SO(3)$ and $U = + \sqrt{F^T F}$. Further, $U = Q^T diag ( v_1, v_2, v_3) Q$ where
$Q$ is the orthogonal matrix of eigenvectors and $v_1,v_2,v_3$ are the eigenvalues of $U$.
The properties (\ref{FRISP}) of isotropy and frame-indifference imply
\begin{equation*}
\begin{aligned}
 W(F) &=\Bar{\Bar G}\BBR{\diagmtx{v_1}{v_2}{v_3},\diagmtx{v_2v_3}{v_1 v_3}{v_1 v_2}, v_1 v_2 v_3}  \\
 &=: \Bar{G}\BBR{v_1,v_2,v_3,v_2 v_3,v_1 v_3, v_1 v_2, v_1 v_2 v_3}
 \end{aligned}
\end{equation*}
where $\Bar{G} (\Xi)$ is a convex function of
$\Xi=(\xi)_{i=1\dots 7} \in \RR^7$.

For radial motions the singular values are $v_1 = w_R$, $v_2 = v_3 = \frac{w}{R}$.
For reasons related to the null-Lagrangian structure of an associated variational problem
(outlined in the following section)
the stored energy will be expressed in the form
\begin{equation}\label{PXENERGY}
\begin{aligned}
W(\gr{y}) &= \Phi\BBR{w_R,\f{w}{R},\f{w}{R}}
\\
&=  \Bar{G}\BBR{w_R,\f{w}{R},\f{w}{R}, \big (\f{w}{R} \big )^2 ,
w_R \big ( \f{w}{R} \big ) , w_R\big ( \f{w}{R} \big ) , w_R\big (\f{w}{R} \big )^2 }
\\
&=  G\BBR{ \Omega\BBR{ \BBR{w_R,\f{w}{R},\f{w}{R}} ;R }  ;  R  }
\end{aligned}
\end{equation}
where $\Omega$ and $G$ are inhomogeneous functions defined by
\begin{align}
\label{OMEGADEF}
\Omega(V;R) &:=\BBR{v_1,v_2,v_3,v_2v_3R,v_1v_3R,v_1v_2R,v_1v_2v_3R^2} \, ,
\\
\label{GDEF}
G(\Xi;R) &:=\Bar{G}\BBR{\xi_1,\xi_2,\xi_3,\xi_4/R,\xi_5/R,\xi_6/R,\xi_7/R^2} \, ,
\end{align}
$V=(v_i)_{i=1 \dots 3}\in \RR^3$ and $\Xi=(\xi)_{i=1
\dots 7} \in \RR^7$.
The convexity hypothesis on $\Bar{\Bar{G}}$ implies that  $G(\Xi ; R)$
is convex as a function of $\Xi \in \RR^7$.
In summary,
\begin{align}
\label{COMP}
W(\gr{y}) = \Phi\BBR{w_R,\f{w}{R},\f{w}{R}} &= G(\Omega(\Gamma;R);R),
\\
\text{where } \quad
\label{GAMMADEF}
\Gamma &= \BBR{w_R,\f{w}{R},\f{w}{R}} \, .
\end{align}
For simplicity of notation, we henceforth suppress the dependence on $R$ and write
$\Omega(V)=\Omega(V;R)$ and $G(\Xi)=G(\Xi;R)$.

Equation \eqref{RPDE} can be expressed in the form
\begin{equation}\label{3rpde}
\begin{aligned}
        R^{2}\partial_{t } v
        &=\pd{}{R}\BBR{  R^{2}\Phi_{,1} \big (  w_{R} ,\f{w}{R} , \f{w}{R} \big ) }
         -R \BBR{ \Phi_{,2} + \Phi_{,3} }  \big (  w_{R} ,\f{w}{R} , \f{w}{R} \big ) \, .
         \\
         \del_t w &= v \, ,
 \end{aligned}
\end{equation}
The latter
formally satisfies the conservation of mechanical energy identity
\begin{equation}
\label{ENLENR}
\begin{aligned}
\partial_t \Big (   R^{2}  \big ( \f{v^2}{2} + \Phi  \BBR{w_R,w/R, w/R}   \big ) \Big )
 = \partial_{R} \BBR{ R^{2} v  \, \Phi_{,1}\BBR{w_R,w/R,w/R} } \, .
 \end{aligned}
\end{equation}
The mechanical energy and the associated energy flux provide an entropy-entropy flux pair for \eqref{3rpde} but the
entropy is not in general convex.
Using \eqref{COMP}-\eqref{GAMMADEF},  the derivatives
$\Phi_{, j}$ are expressed as
\begin{equation*}
\Phi_{,j}(v_1,v_2,v_3)  = \f{\partial}{\partial v_j}
G(\Omega(V))=\pd{G}{\xi_i}(\Omega(V)) \pd{\Omega^i}{v_j}(V),
\end{equation*}
and \eqref{3rpde}$_1$ is written as
\begin{equation}
\label{PXPDE}
\begin{aligned}
R^2\,\partial_t v &=  \;
\partial_{R} \BBR{R^2 \pd{G}{\xi_i}(\Omega(\Gamma)) \pd{\Omega^i}{v_1}(\Gamma) }\\
& \qquad - R \, \pd{G}{\xi_i}(\Omega(\Gamma)) \BBR{\pd{\Omega^i}{v_2}(\Gamma) + \pd{\Omega^i}{v_3}(\Gamma)} \, .
\end{aligned}
\end{equation}


\section{Null-Lagrangians and extensions of polyconvex radial elasticity}
\label{secnullag}

\subsection{Null-Lagrangians}
An alternative approach to derive \eqref{ENLENR} proceeds
by considering the extrema of the action functional
$$
J[y] = \int_0^T \int_0^1 \BBR{  \f{1}{2} w_t^2 - \Phi \big (  w_R,\f{w}{R} , \f{w}{R} \big ) } R^2 \, dR dt
$$
and deriving \eqref{RPDE} (for $n=3$)  as the associated Euler-Lagrange equations. This provides
a connection with the calculus of variations.

Consider the functional associated to the equilibrium problem
\begin{equation*}
I[w]=\int\limits_{0}^{1} \Psi \BBR{w_R,w/R,w/R \, ;  R}dR \, .
\end{equation*}
We ask for which integrands $\Psi\BBR{v_1,v_2,v_3;R}:\RR^4 \to \RR$
the functional $I$ admits zero variational
derivatives,  $\frac{\delta I}{\delta w} = 0$;
such integrands are called {\it null Lagrangians} and they satisfy
the Euler-Lagrange equation
\begin{equation}\label{NLAGR}
- \partial_{R} \BBR{\Psi_{,1}} + R^{-1} \BBR{ \Psi_{,2} +
\Psi_{,3}}=0 \quad \text{for all functions} \;  w(R) \, .
\end{equation}
If $w = w(R,t)$ also depends on time, the evolution of a null Lagrangian $\Psi$ is described by
\begin{equation}\label{NLEVOL}
\partial_t \Psi = \partial_{R} \BBR{\Psi_{,1}\,
\partial_t w}.
\end{equation}
where $\Psi$ and $\Psi_{,i}$ are evaluated
at $\BBR{w_R, w/R, w/R, R}$.
\par\bigskip

It is easily verified that $\Psi(v_1,v_2,v_3 ; R)$ selected  by
\begin{equation*}
v_1, \;\;\; v_1 v_2 R, \;\;\; v_1 v_3 R,   \txq{or} v_1 v_2 v_3
R^2
\end{equation*}
are null-Lagrangians. Applying
(\ref{NLAGR}) to $ \Omega^i$, $i=1,5,6,7$,  defined by
\eqref{OMEGADEF} we get
\begin{equation}\label{NLAGROMEGA}
-\partial_{R} \BBR{ \Omega^{i}_{,1}(\Gamma)} + R^{-1} \BBR{
\Omega^i_{,2}(\Gamma) + \Omega^i_{,3}(\Gamma)}=0, \;\;\;\;
i=1,5,6,7,
\end{equation}
with $\Gamma = (w_R , w/R, w/R)$ defined by (\ref{GAMMADEF}).
\par\bigskip

\subsection{A symmetrizable extension} \label{extensionone}
The null-Lagrangian structure is used in \cite{Tz2} to embed the equations of 3-d elastodynamics
to a hyperbolic system endowed with a convex entropy, and to construct a variational approximation scheme
for the problem. We follow this procedure in order to achieve an augmented system for
radial elastodynamics. The evolution in time of
\begin{equation}
\Omega(\Gamma)=\BBR{w_R,w/R,w/R,w^2/R,w_R w,w_R w,w_R w^2}
\end{equation}
gives
\begin{equation}\label{OMEGAEVOL}
\begin{aligned}
\partial_t \, \Omega^1(\Gamma) &= \partial_t \BBR{w_R}    = \partial_R v    = \partial_R \BBR{\Omega^1_{,1}(\Gamma)v}\\
\partial_t \, \Omega^i(\Gamma) &= \partial_t \BBR{w/R}    = v/R             = R^{-1} \BBR{\Omega^i_{,2}(\Gamma)+\Omega^i_{,3}(\Gamma)}v \quad \text{ for $i=2,3$}
\\
\partial_t \, \Omega^4(\Gamma) &= \partial_t \BBR{w^2/R}  = 2(w/R)v         = R^{-1} \BBR{\Omega^4_{,2}(\Gamma)+\Omega^4_{,3}(\Gamma)}v \\
\partial_t \, \Omega^i(\Gamma) &= \partial_t \BBR{w_Rw}   = \partial_R(wv)  = \partial_R \BBR{\Omega^i_{,1}(\Gamma)v}  \qquad \text{ for $i=5,6$}
\\
\partial_t \, \Omega^7(\Gamma) &= \partial_t \BBR{w_Rw^2} = \partial_R(w^2v)= \partial_R \BBR{\Omega^7_{,1}(\Gamma)v}.\\
\end{aligned}
\end{equation}
Note that (\ref{OMEGAEVOL})$_{1,5,6,7}$ are precisely the equations (\ref{NLEVOL}) describing
the evolution of null Lagrangians. By contrast, (\ref{OMEGAEVOL})$_{2,3,4}$ describe the evolution of
lower-order terms and do not have the structure of (\ref{NLEVOL}).
\par\medskip

Equations \eqref{OMEGAEVOL} and \eqref{PXPDE}
motivate an extension of radial elasticity :
\begin{align}
\label{EXTSYSPOS}
&\left\{
\begin{aligned}
R^2\partial_t v &= \partial_{R} \BBR{R^2 \pd{G}{\xi_i}(\Xi)
\pd{\Omega^i}{v_1}(\xi) }- R\pd{G}{\xi_i}(\Xi)
\BBR{\pd{\Omega^i}{v_2}(\xi)
 +\pd{\Omega^i}{v_3}(\xi)}\\
\partial_t \xi_i &= \partial_R \BBR{\Omega^i_{,1}(\xi)v }
\;\;\;i=1,5,6,7\\
\partial_t \xi_i &= R^{-1}
\BBR{\Omega^i_{,2}(\xi)+\Omega^i_{,3}(\xi)}v \;\;\; i=2,3,4
\\
\xi_1 &= \del_R ( R \xi_2 ) \, , \quad \xi_2 = \xi_3
\end{aligned}
\right.
\\
\label{extsysconst}
&\qquad  \xi_2(1) =\xi_3(1)=\lambda,\; \xi_{2},\xi_3 \geqslant 0, \;\xi_7 >
0, \;\; (R,t) \in (0,1) \times [0,\infty),
\end{align}
System \eqref{EXTSYSPOS} describes the evolution of the vector $(v,\Xi)$,  where
 $\Xi \in \RR^7$ and $\xi = ( \xi_1, \xi_2, \xi_3 )$ are the first three components of $\Xi$.
 \par\medskip

 The extension has the following properties:

 \medskip
 \noindent
(i)  The constraint \eqref{EXTSYSPOS}$_4$ enforces that $\xi$ is of the form
$\xi = (w_R , w/R, w/R)$ for some function $w(R,t)$
(similarly to  $\Gamma$ in \eqref{GAMMADEF}).
Moreover, \eqref{EXTSYSPOS}$_4$ is an involution: if it is satisfied
for the initial data, the constraint is propagated and is satisfied for all times.

\smallskip
\noindent
(ii) If $\Xi(\cdot,0)=\Omega(\Gamma^0)$ where
$\Gamma^0=\BBR{f',f/R,f/R}$ for some $f=f(R)$, then $\Xi(R,t)$  retains the same format
for all times, {\it i.e.} there exists $w$ such that $\Xi=\Omega(\Gamma)$ where
$\Gamma=(w_R,w/R,w/R)$.  In other
words, radial elasticity (\ref{3rpde}) can be viewed as a
constrained evolution of \eqref{EXTSYSPOS}.

\smallskip
\noindent
(iii) The enlarged system admits an entropy pair
\begin{equation}\label{EXTENTRPOS}
\partial_t \BBR{ R^2 \BBR{ \f{v^2}{2} + G(\Xi) }} - \partial_{R} \BBR{R^2 \, \pd{G}{\xi_i}(\Xi)
\, \pd{\Omega^i}{v_1}(Z) \,v}= 0 \, ,
\end{equation} with strictly convex entropy
\begin{equation}
\eta(v,\Xi)=\f{v^2}{2} + G(\Xi).
\end{equation}
Let us remark that $\eta$ is not an entropy in the usual sense of
the theory of conservation laws: the identity
(\ref{EXTENTRPOS}) is based on the constraint
(\ref{EXTSYSPOS})$_4$ together with the property (\ref{NLAGROMEGA}) of null
Lagrangians.

\subsection{An alternative extension with a convex entropy}\label{extensiontwo}

System \eqref{EXTSYSPOS} provides an extension of radial
elasticity that is endowed with a convex entropy. Concerning the
objective of achieving a variational approximation, it has the
drawback that the constraint \eqref{extsysconst} of positivity
for the variables $\xi_2,\xi_3$ and $\xi_7$ is not preserved at the level of
time-step approximations. Although one can
control the positivity of $\xi_7$ (the augmented variable standing for the determinant),
it is not possible to control the positivity of $\xi_2$,$\xi_3$.
There are also difficulties in proving that minimizers satisfy the
corresponding Euler-Lagrange equations, the time-discretized
system associated to \eqref{EXTSYSPOS}.

For this reason, we develop an alternative extension by combining
the evolution of null Lagrangians with a change of variables used
in  Ball \cite{Ba} for the equilibrium problem. This extension
induces a variational approximation scheme that preserves the
positivity of determinants.

The stored energy $\Phi$ is expressed in the form
\begin{equation}
\label{storeden1}
\begin{split}
 \Phi\BBR{v_1,v_2,v_3}
&=
\Bar{G}\BBR{ v_1,v_2,v_3,v_2v_3,v_1v_3,v_1v_2, v_1v_2v_3}
\\
&=  G(\Omega(V ; \rho) \, ; \, \rho)
\end{split}
\end{equation}
where $\Omega$ and $G$ are nonhomogeneous functions of $\rho$ that are redefined so that
\begin{align}
\label{OMEGADEF2}
\Omega(V;\rho) &:=
\BBR{v_1,v_2^3,v_3^3,v_2v_3\rho^{1/3},v_1v_3\rho^{1/3},v_1v_2\rho^{1/3},v_1v_2v_3\rho^{2/3}}
\\
\label{GDEF2}
G(\Xi;\rho) &:=
\Bar{G}\BBR{\xi_1,\xi_2^{1/3},\xi_3^{1/3},\xi_4/\rho^{1/3},\xi_5/\rho^{1/3},\xi_6/\rho^{1/3},\xi_7/\rho^{2/3}}.
\end{align}
It is now assumed that $G(\Xi ; \rho)$ is a convex function of $\Xi$; this is a somewhat stronger
hypothesis than polyconvexity (which is convexity of $\bar G$) because of the definition of
$\Omega^i (V ; \rho)$, $i=2,3$, in \eqref{OMEGADEF2}. In the sequel  any explicit $\rho$-dependence will
be suppressed.

\subsubsection{A change of variables}
Following \cite{Ba} we perform the change of variables
\begin{equation}
\rho=R^3 \txq{and} \alpha=w^3.
\end{equation}
Then $\Gamma = (w_R , w/R, w/R )$ is expressed as
\begin{equation}\label{GAMMADEF2}
\Gamma =
(\alpha_{\rho}(\rho/\alpha)^{2/3},(\alpha/\rho)^{1/3},(\alpha/\rho)^{1/3})
\end{equation}
and the stored energy reads
\begin{equation}
\begin{split}
&W(\gr{y}) = \Phi\BBR{\alpha_{\rho}(\rho/\alpha)^{2/3},(\alpha/\rho)^{1/3},(\alpha/\rho)^{1/3}}\\
&=  G(\Omega(\Gamma ; \rho) \, ; \, \rho)
\end{split}
\end{equation}
where $\Omega$ and $G$ are defined in \eqref{OMEGADEF2}, \eqref{GDEF2},
and  $G(\cdot ; \rho)$ is convex.

The system \eqref{3rpde} takes the form
\begin{equation}\label{PXPDE2}
\left\{
\begin{aligned}
\partial_t v &=
\partial_{\rho} \BBR{3\rho^{2/3} \pd{G}{\xi_i}(\Omega(\Gamma)) \pd{\Omega^i}{v_1}(\Gamma) }\\
& - \rho^{-1/3} \, \pd{G}{\xi_i}(\Omega(\Gamma)) \BBR{\pd{\Omega^i}{v_2}(\Gamma) + \pd{\Omega^i}{v_3}(\Gamma)}\\
\partial_t (\alpha^{1/3}) &=v\\
\alpha(1)&=\lambda,\;\alpha \geqslant 0, \; \alpha_{\rho} > 0,
\;\; (R,t) \in (0,1) \times [0,\infty).
\end{aligned}\right.
\end{equation}
with the last inequalities encoding the constraints for solutions to represent elastic
motions.
In the new variables, by \eqref{OMEGADEF2},
\begin{equation}
\Omega(\Gamma)=\BBR{\f{\alpha_{\rho}}{\alpha^{2/3}}\rho^{2/3},\f{\alpha}{\rho},\f{\alpha}{\rho},\f{\alpha^{2/3}}{{\rho}^{1/3}},\f{\alpha_{\rho}}{\alpha^{1/3}}\rho^{2/3},\f{\alpha_{\rho}}{\alpha^{1/3}}\rho^{2/3},\alpha_{\rho}\rho^{2/3}}
\end{equation}
and, using  \eqref{PXPDE2}$_2$, we compute
\begin{equation}\label{OMEGAEVOL2}
\begin{aligned}
\partial_t \, \Omega^1(\Gamma) &= \partial_t \BBR{3\rho^{2/3}\partial_{\rho}(\alpha^{1/3})} = 3\rho^{2/3}\partial_{\rho}v\\
\partial_t \, \Omega^i(\Gamma) &= \partial_t \BBR{\alpha/\rho} = 3\alpha^{2/3}v/\rho
\qquad i=2,3  \\
\partial_t \, \Omega^4(\Gamma) &= \partial_t \BBR{\alpha^{2/3}/{\rho}^{1/3}} = 2\alpha^{1/3}v/\rho^{1/3}\\
\partial_t \, \Omega^i (\Gamma) &= \partial_t \BBR{(3/2)\rho^{2/3}\partial_{\rho}(\alpha^{2/3})} = 3\rho^{2/3}\partial_{\rho}\BBR{\alpha^{1/3}v}
\quad i=5,6 \\
\partial_t \, \Omega^7(\Gamma) &= \partial_t \BBR{\alpha_{\rho}\rho^{2/3}} = 3\rho^{2/3}\partial_{\rho}(\alpha^{2/3}v) \, . \\
\end{aligned}
\end{equation}\\
These identities are summarized in two groups
\begin{equation}\label{OMEGAEVOL2SHORT}
\begin{aligned}
\partial_t \Omega^i(\Gamma) &=
3\rho^{2/3}\partial_{\rho}(\Omega^i_{,1}(\Gamma)v), \;\;\;
i=1,5,6,7 \, , \\
\partial_t \Omega^i(\Gamma) &=
\rho^{-1/3}(\Omega^i_{,2}(\Gamma)+\Omega^i_{,3}(\Gamma))v \, ,
\;\;\;i=2,3,4 \, ,
\end{aligned}
\end{equation}
the former representing the evolution of null-Lagrangians and the latter the evolution of
lower order terms. The identities (\ref{NLAGROMEGA}) satisfied by null-Lagrangians
become
\begin{equation}\label{NLAGROMEGA2}
-3\rho^{2/3} \partial_{\rho} \BBR{ \Omega^{i}_{,1}(\Gamma)} +
{\rho}^{-1/3} \BBR{ \Omega^i_{,2}(\Gamma) +
\Omega^i_{,3}(\Gamma)}=0, \;\;\;\; i=1,5,6,7.
\end{equation}
\par\bigskip

\subsubsection{The augmented system}

Next, consider the augmented system
\begin{equation}\label{EXTSYS}
\left\{
\begin{aligned}
\partial_t v &= \partial_{\rho} \BBR{3\rho^{2/3}
G_{,i}(\Xi)\Omega^i_{,1}(\Gamma) }-\rho^{-1/3}G_{,i}(\Xi)
\BBR{\Omega^i_{,2}(\Gamma)
 +\Omega^i_{,3}(\Gamma)}\\
 \partial_t \alpha^{1/3} &= v \\
\partial_t \xi_i &= 3\rho^{2/3}\partial_{\rho} \BBR{\Omega^i_{,1}(\Gamma)v
},
\;\;\;i=1,5,6,7\\
\partial_t \xi_i &= {\rho}^{-1/3}
\BBR{\Omega^i_{,2}(\Gamma)+\Omega^i_{,3}(\Gamma)}v, \;\;\; i=2,3,4
\\
\xi_1 &= 3 \rho^{2/3} \partial_\rho \alpha^{1/3}
\end{aligned}\right.
\end{equation}
The system
\eqref{EXTSYS}$_1$-\eqref{EXTSYS}$_4$ describes the evolution of
the vector $(v, \alpha, \Xi)$ subject to the constraint
\eqref{EXTSYS}$_5$. It has the properties:

\medskip
\noindent
\begin{description}
\item[(a)] The constraint \eqref{EXTSYS}$_5$ is propagated by the evolution from the initial data,
since $\del_t (\xi_1 - 3 \rho^{2/3} \partial_\rho \alpha^{1/3}) = 0$.
We may thus write $\Omega(\Gamma)$, with $\Gamma$ as in \eqref{GAMMADEF2},
and still think of \eqref{EXTSYS} as a first order system.

\item[(b)]  If $\Xi(\cdot,0)=\Omega(\Gamma^0)$ with $\Gamma^0=\BBR{
f'(\rho / f)^{2/3},(f/\rho)^{1/3},(f/\rho)^{1/3}}$ for
$f=f(\rho)$, it remains in this form $\forall t$, {\it i.e.}
there exists $\alpha(\rho,t)$ such that  $\Gamma$ defined by
(\ref{GAMMADEF2}) satisfies $\Gamma(.,0)=\Gamma^0$ and
$\Xi=\Omega(\Gamma)$ $\forall t$. In other words, radial
elasticity (\ref{RPDE}) can be viewed as a constrained
evolution of (\ref{EXTSYS}).

\item[(c)] The enlarged system admits an entropy pair
\begin{equation}\label{EXTENTR}
\partial_t \BBR{ \f{v^2}{2} + G(\Xi) } - \partial_{\rho} \BBR{3 \rho^{2/3} \, G_{,i}(\Xi)
\, \Omega_{,1}^i(\Gamma) \,v}= 0
\end{equation} with (for convex $G$) strictly convex entropy
$\eta(v,\Xi)=\f{v^2}{2} + G(\Xi)$.
\end{description}

At this point we set
\begin{align}
&\qquad \qquad  \beta=\alpha_{\rho}/\alpha^{2/3} \, , \quad \gamma=\alpha^{2/3} \, ,
\nonumber
\\
\label{XIDEF}
&\Xi=\;\BBR{\beta \rho^{2/3}, \f{\alpha}{\rho},\f{\alpha}{\rho},
\f{\gamma} {\rho^{1/3}}, \f{3\gamma_{\rho}}{2}\rho^{2/3},
\f{3\gamma_{\rho}}{2}\rho^{2/3},\alpha_{\rho}\rho^{2/3}} \, ,
\end{align}
and proceed to simplify the extended system working with
$\alpha,\beta,\gamma,v$ as the independent variables.

Taking a closer look at the extended system we see that
$\xi_2=\xi_3$ by construction and hence equations
(\ref{EXTSYS})$_2$, $i=2,3$ are identical. Moreover,
\begin{equation}
\begin{aligned}
\partial_t \xi_2 = 3\alpha^{2/3}v/\rho \;\;\; &\Rightarrow
\;\;\; \partial_t
\xi_7=\rho^{2/3}\partial_{\rho}(\rho \,\partial_t \xi_2),\\
\partial_t \xi_4 = 2\alpha^{1/3}v/\rho^{1/3} \;\;\; &\Rightarrow
\;\;\; \partial_t \xi_5=\partial_t
\xi_6=\f{3}{2}\rho^{2/3}\partial_{\rho}\BBR{\rho^{1/3}
\partial_t\xi_4}.
\end{aligned}
\end{equation}
Hence (\ref{EXTSYS}) is overdetermined and
extra equations (\ref{EXTSYS})$_2$, $i=5,6,7$ and
(\ref{EXTSYS})$_3$, $i=3$ can be excluded. In
explicit form the extension is written as
\begin{equation}\label{EXPLSYS}
\left\{
\begin{aligned}
\partial_t \,v &=
\partial_{\rho} \BBR{ 3\rho^{2/3} \, G_{,i}(\Xi) \, \Omega_{,1}^i(\Gamma)} -
\rho^{-1/3}\,
G_{,i}(\Xi) \, \BBR{\Omega_{,2}^i(\Gamma)+\Omega_{,3}^i(\Gamma)}\\
\partial_t \beta &= \partial_{\rho}(3v)   \\
\partial_t \alpha &= 3 \alpha^{2/3}v   \\
\partial_t \gamma &= 2 \alpha^{1/3}v    \\
\alpha(1)&=\lambda,\;\alpha \geqslant0, \;\;\alpha_{\rho} >0, \;\;
(\rho,t) \in (0,1)\times[0,\infty),
\end{aligned}
\right.
\end{equation} \\
where from (\ref{EXPLSYS})$_3$ and (\ref{EXPLSYS})$_4$ we can
derive the excluded equations
\begin{equation}\label{EXPLSYSEXTRA}
\begin{split}
\partial_t \alpha_{\rho} &= \partial_{\rho} (3 \alpha^{2/3}v) \\
\partial_t \gamma_{\rho} &= \partial_{\rho} (2 \alpha^{1/3}v).  \\
\end{split}
\end{equation}\par


\section{Variational Approximation Scheme}
\label{secvarapp}

In this section we introduce a variational approximation scheme
for the radial equation of elastodynamics.  The general approach
is to discretize the extended system by use of implicit-explicit
scheme. \par

Successive iterates are constructed by discretizing (\ref{EXTSYS})
as follows: Given the $(j-1)^{th}$ iterate
$(\alpha_0,\beta_0,\gamma_0,v_0)$ with $\alpha_0(\rho)\geqslant0$
and ${\alpha_0'(\rho)}>0$, $\rho\in(0,1)$,  we define
$\Xi^0=(\xi_i^0)_{i=1}^{7}$ by
\begin{equation}\label{XIZERODEFDISC}
\Xi^0(\rho)=\;\BBR{\beta_0 \rho^{2/3},
\f{\alpha_0}{\rho},\f{\alpha_0}{\rho}, \f{\gamma_0} {\rho^{1/3}},
\f{3{\gamma_0'}}{2}\rho^{2/3},
\f{3{\gamma_0'}}{2}\rho^{2/3},{\alpha_0'}\rho^{2/3}}
\end{equation}
and construct the $j^{th}$ iterate $(\alpha,\beta,\gamma,v)$, with
corresponding $\Xi=(\xi_i)_{i=1}^{7}$ defined by
\begin{equation}\label{XIDEFDISC}
\Xi(\rho)=\;\BBR{\beta \rho^{2/3},
\f{\alpha}{\rho},\f{\alpha}{\rho}, \f{\gamma} {\rho^{1/3}},
\f{3{\gamma'}}{2}\rho^{2/3},
\f{3{\gamma'}}{2}\rho^{2/3},{\alpha'}\rho^{2/3}} \, ,
\end{equation}
by solving
\begin{equation}\label{EXTDISCSYS}
\left\{
\begin{aligned}
(v-v_0)/h &= \f{d}{d\rho} \BBR{3\rho^{2/3} G_{,i}(\Xi)
\Omega^i_{,1}(\Gamma^0) }\\
&-\rho^{-1/3}G_{,i}(\Xi) \BBR{\Omega^i_{,2}(\Gamma^0)
 +\Omega^i_{,3}(\Gamma^0)}\\
(\xi_i-\xi^0_i)/h &= 3\rho^{2/3}\f{d}{d\rho}
\BBR{\Omega^i_{,1}(\Gamma^0)v },
\;\;i=1,5,6,7\\
(\xi_i-\xi^0_i)/h &= {\rho}^{-1/3}
\BBR{\Omega^i_{,2}(\Gamma^0)+\Omega^i_{,3}(\Gamma^0)}v, \;\;
i=2,3,4
\\
\xi_2(1)&=\xi_3(1)=\lambda,\; \xi_{2},\xi_{3} \geqslant 0, \;\xi_7
> 0, \;\; \rho\in (0,1),
\end{aligned}\right.
\end{equation}
where
\begin{align}
\label{GAMMADEFDISC}
\Gamma &=
(\alpha'(\rho/\alpha)^{2/3},(\alpha/\rho)^{1/3},(\alpha/\rho)^{1/3}) \, ,
\\
\label{GAMMAZERODEFDISC}
\Gamma^0 &=
(\alpha_0'(\rho/\alpha_0)^{2/3},(\alpha_0/\rho)^{1/3},(\alpha_0/\rho)^{1/3}).
\end{align}\par\medskip

As in the continuous case the discrete system (\ref{EXTDISCSYS})
is overdetermined with extra equations
\begin{equation}\label{EXPLDISCSYSEXTRA}
\begin{aligned}
\BBR{\alpha_{\rho} - {\alpha_0}_{\rho}}/{h} &= \f{d}{d\rho}\BBR{3{\alpha_0}^{2/3}v} \, , \\
\BBR{\gamma_{\rho} - {\gamma_0}_{\rho}}/{h} &= \f{d}{d\rho}\BBR{2 {\alpha_0}^{1/3}v} \, , \\
\end{aligned}
\end{equation}
corresponding to (\ref{EXTDISCSYS})$_2$, $i=5,6,7$. Excluding them
from the system  above we get
\begin{equation}\label{EXPLDISCSYS}
\left\{
\begin{aligned}
\BBR{v-v_0}/h &=
\f{d}{d\rho}\BBR{ 3\rho^{2/3} \, G_{,i}(\Xi) \, \Omega_{,1}^i(\Gamma^0)}\\
 &-\rho^{-1/3}\,G_{,i}(\Xi) \BBR{\Omega_{,2}^i(\Gamma^0)+\Omega_{,3}^i(\Gamma^0)}\\
\BBR{\beta  - \beta_0 }/{h} &= \f{d}{d\rho}(3v) \\
\BBR{\alpha - \alpha_0}/{h} &= 3 {\alpha_0}^{2/3}v    \\
\BBR{\gamma - \gamma_0}/{h} &= 2 {\alpha_0}^{1/3}v    \\
\alpha(1)&=\lambda, \;\alpha \geqslant0, \;\;\alpha' >0, \;\;
\rho\in (0,1) \, .
\end{aligned}
\right.
\end{equation}
Note that equations (\ref{EXPLDISCSYSEXTRA}) can be
derived from (\ref{EXPLDISCSYS})$_{3,4}$.
\par\medskip

Time-step approximations capture a subtle form of dissipation associated with the
underlying variational structure and the convexity of the entropy, \cite{Tz,Tz2}.
Indeed, solutions of \eqref{EXPLDISCSYS} satisfy a discrete entropy inequality:
To see that, consider a smooth solution $(\Xi,v)$ of (\ref{EXTDISCSYS}) associated to
smooth data $(\Xi^0,v_0)$ given by (\ref{XIZERODEFDISC}). Multiplying
(\ref{EXTDISCSYS})$_1$ by $v$ we get
\begin{equation}\label{MULTDISCV}
\begin{split}
\f{v(v-v_0)}{h} &+  G_{,i}(\Xi)  \BBR{ 3\rho^{2/3}
\Omega^i_{,1}(\Gamma^0)\f{dv}{d\rho}+
\rho^{-1/3}\BBR{\Omega^i_{,2}(\Gamma^0) +
\Omega^i_{,3}(\Gamma^0)}v} \\
&=\f{d}{d\rho} \BBR{3\rho^{2/3}
G_{,i}(\Xi)\Omega^i_{,1}(\Gamma^0)v}.
\end{split}
\end{equation}
Then denoting
\begin{equation}\label{ADEF}
A_i = 3\rho^{2/3} \Omega^i_{,1}(\Gamma^0)\f{dv}{d\rho}+
\rho^{-1/3}\BBR{\Omega^i_{,2}(\Gamma^0) +
\Omega^i_{,3}(\Gamma^0)}v, \;\;\;i=1,\dots,7
\end{equation}
we claim
\begin{equation}\label{APROP}
A_i = \f{\xi_i-\xi^0_i}{h}.
\end{equation}
Indeed, for $i=2,3,4$ we have $\Omega^i_{,1}=0$ and hence
(\ref{EXTDISCSYS})$_3$ and (\ref{ADEF}) imply (\ref{APROP}). For
$i=1,5,6,7$ by the properties (\ref{NLAGROMEGA2}) of null
Lagrangians and (\ref{EXTDISCSYS})$_2$ we get
\begin{equation}\label{XIDISCEVOL2}
\begin{split}
A_i &= v \BBR{ -
3\rho^{2/3}\f{d}{d\rho}\BBR{\Omega^i_{,1}(\Gamma^0)} + \rho^{-1/3}
\BBR{\Omega^i_{,2}(\Gamma^0)+\Omega^i_{,3}(\Gamma^0)}}\\ &+
3\rho^{2/3}\f{d}{d\rho}\BBR {\Omega^i_{,1}(\Gamma^0) v}
 = \BBR{\xi_i-\xi^0_i}/h, \;\;\; i=1,5,6,7.
\end{split}
\end{equation}
Thus (\ref{MULTDISCV}) and (\ref{APROP}) imply
\begin{equation*}
\f{1}{h}\BBR{v(v-v_0)+G_{,i}(\Xi)\BBR{\xi_i-\xi^0_i}}=\f{d}{d\rho}
\BBR{3\rho^{2/3} G_{,i}(\Xi)\Omega^i_{,1}(\Gamma^0)v}.
\end{equation*}
Now, we denote $\Theta=(v,\Xi)$ and $\Theta^0=(v_0,\Xi^0)$. Then
$\eta = 1/2 v^2 + G(\Xi)$ satisfies
\begin{equation*}
\f{1}{h}D\eta \cdot \BBR{\Theta-\Theta^0}-\f{d}{d\rho}
\BBR{3\rho^{2/3} G_{,i}(\Xi)\Omega^i_{,1}(\Gamma^0)v}=0.
\end{equation*}
For $G$ convex  the
following identity holds
\begin{equation}\label{DISCENTR}
\f{\eta(\Theta)-\eta(\Theta^0)}{h}-\f{d}{d\rho} \BBR{3\rho^{2/3}
G_{,i}(\Xi)\Omega^i_{,1}(\Gamma^0)v} \leqslant 0.
\end{equation}\par


\begin{Remark} \rm
We have not studied in this article the convergence as the time-step $h \to 0$.
For the three-dimensional elasticity equations this process produces measure-valued solutions
\cite{Tz2} while for one-dimensional elasticity it gives entropy weak solutions \cite{Tz}.
In the present case we would expect to obtain weak solutions, but the compactness properties of
\eqref{SYSTEMINTRO} are not at present sufficiently understood.
There are two differences of \eqref{SYSTEMINTRO} relative to the well understood compactness theory of one-dimensional elasticity: the
dependence of the stress on lower order terms, and the singularity at $R=0$.
Nevertheless, if the iterates $u^h$, $v^h$ converge strongly, the discrete entropy inequality
\eqref{DISCENTR} gives a weak solution dissipating the mechanical energy.
\end{Remark}


\section{Existence of minimizers}
\label{seceximin}

Henceforth, we consider stored-energy functions \eqref{storeden1} of the form
\begin{equation}\label{PHIDEF}
\begin{split}
&\Phi(v_1,v_2,v_3)  = \bar G (v_1, v_2,v_3,_2v_3,v_1v_3,v_1v_2,v_1v_2v_3)
\\
&= \varphi(v_1) + \varphi(v_2) + \varphi(v_3)
+g(v_2v_3) + g(v_1v_3) + g(v_1v_2) + h(v_1v_2v_3).
\end{split}
\end{equation}
Then, the function $G$ defined in (\ref{GDEF2}) reads
\begin{equation}\label{GPART}
\begin{aligned}
G(\Xi;\rho) &=  \varphi(\xi_1) +
\varphi\BBR{\xi_2^{1/3}}+\varphi\BBR{\xi_3^{1/3}}
\\&+g\BBR{\xi_4\rho^{1/3}}
+ g\BBR{\xi_5\rho^{1/3}} + g\BBR{\xi_6\rho^{1/3}} +
h(\xi_7\rho^{2/3}).\\
\end{aligned}
\end{equation}
Now, define $\psi(x)=\varphi(x^{1/3})$.  Then, with $\Xi$ defined
in (\ref{XIDEFDISC}), the above is expressed by
\begin{equation}\label{GPARTEXPL}
\begin{split}
G(\Xi) =\varphi(\beta\rho^{2/3})
+2\psi\BBR{{\alpha}/{\rho}}+g\BBR{\gamma/\rho^{2/3}}+2g\BBR{3\gamma'\rho^{1/3}/2}+h(\alpha').
\end{split}
\end{equation}\par\bigskip

We place the following
assumptions on the functions $\varphi$, $\psi$, $g$, $h$ appearing
above:
\begin{itemize}
\item[(A1)] $\lim_{\delta \to 0+} h(\delta) = \lim_{\delta \to
+\infty}h(\delta)/\delta=+\infty$; \item[(A2)] $\varphi,\psi, g
\in C^2(\RR)$ and $h \in C^2(\RR_+)$ satisfy
\begin{equation}
\varphi, \psi, g, h , \varphi'', \psi'', g''\geqslant 0 \txq{and}
h'' > 0;
\end{equation}
\item[(A3)] For $1<p,q<\infty$ and some constants $c_1,c_2>0$
\begin{equation}
\lim_{x\to\infty} \f{\varphi(x)}{|x|^{3p}} = \lim_{x\to\infty}
\f{\psi(x)}{|x|^p} = c_1 \, , \; \lim_{x\to\infty}
\f{g(x)}{|x|^q}=c_2;
\end{equation}
\item[(A4)] For $1<p,q<\infty$ as in (A3) and $C_1,C_2,C_3 >0$
\begin{equation}
\limsup_{x\to\infty} \f{ |\varphi'(x)|}{|x|^{3p-1}} \le C_1 \, \; \limsup_{x\to\infty}
\f{|\psi'(x)|}{|x|^{p-1}} \le C_2 \, , \;  \limsup_{x\to\infty}
\f{|g'(x)|}{|x|^q}\le C_3;
\end{equation}
\end{itemize}
In particular,  G is {\it{convex}}.
\par\bigskip
We define spaces of functions on the interval $\rho\in(0,1)$
\begin{equation*}
\begin{aligned}
X_1 &=\BBF{f(\rho)\in W^{1,1}(0,1)      :\;\; f/\rho\in L^p(0,1)},\\
X_2 &=\BBF{f(\rho)\in L^1_{loc}(0,1)    :\;\; f\rho^{2/3}\in L^{3p}(0,1)},\\
X_3 &=\BBF{f(\rho)\in W^{1,1}_{loc}(0,1):\;\; f/\rho^{2/3}\in L^q,\;f'\rho^{1/3}\in L^q(0,1)},\\
Y   &=\BBF{f(\rho)\in W^{1,1}_{loc}(0,1):\;\; f\in L^2,\;f'\rho^{2/3}\in L^{3p}(0,1)},\\
and\\
X&=X_1 \otimes X_2 \otimes X_3 \otimes Y.
\end{aligned}
\end{equation*}
\par\medskip
We fix a parameter $\lambda>0$ and for the initial data
$(\alpha_0,\beta_0,\gamma_0,v_0) \in X$ we require
\begin{equation}\label{IDATAPROP}
\left\{
\begin{aligned}
&\alpha_0(1)=\lambda  \, , \quad  \alpha_0 \geqslant 0 \, , \; \;
{\alpha_0'}>0 \, , \; \tx{a.e.}\;\rho\in(0,1) \, , \\
&\; \; \;  \int\limits_{0}^{1} \, \f{1}{2}{v_0}^2  + G(\Xi^0)
\, d\rho \; < \; \infty \, .
\end{aligned}
\right.
\end{equation}
Consider the problem of minimizing the functional
\begin{eqnarray}
\begin{aligned}
I(\alpha,\beta,\gamma,v)&=\int\limits_{0}^{1} \, \f{1}{2}(v-v_0)^2 + G(\Xi) \, d\rho&&\\
&=\int\limits_{0}^{1} \f{1}{2}(v-v_0)^2+\varphi(\beta\rho^{2/3})
+2\psi\BBR{{\alpha}/{\rho}}\\
&\txq{}\txq{}\txq{}+g\BBR{\gamma/\rho^{2/3}}+2g\BBR{3\gamma'\rho^{1/3}/2}+h(\alpha')\,d\rho
\end{aligned}
\end{eqnarray}
over the admissible set
\begin{eqnarray}
\begin{aligned}
\mathcal{A}_{\lambda}= \{ (\alpha,\beta,\gamma,v)\in X:\,&
\alpha(0) \geqslant 0,\,\alpha(1)=\lambda,\,\alpha'
>0 \;\tx{a.e.} \; and\\
&I(\alpha,\beta,\gamma,v)<\infty,\,\f{\BBR{\beta  - \beta_0 }}{h}=3v',\\
&\f{\BBR{\alpha - \alpha_0}}{h} = 3 {\alpha_0}^{2/3}v,\,
\f{\BBR{\gamma - \gamma_0}}{h} = 2 {\alpha_0}^{1/3}v\}.
\end{aligned}
\end{eqnarray}
We note that $I$ is well-defined for $(\alpha,\beta,\gamma,v)\in
X$ with $\alpha'>0$ a.e. $\rho\in(0,1)$,  though $I$ might be equal to
$\infty$.\par\bigskip

\begin{lmm}
\label{NONEMP}
The admissible set $\mathcal{A}_{\lambda}$ is nonempty.
\end{lmm}
\begin{proof}
Take $(\alpha,\beta,\gamma,v)=(\alpha_0,\beta_0,\gamma_0,0)\in X$.
Then (\ref{IDATAPROP}) implies $\alpha(0)\geqslant 0$,
$\alpha(1)=\lambda$, $\alpha'>0$ a.e. and
\begin{equation*}
I(\alpha,\beta,\gamma,v)=\int\limits_{0}^{1} \, \f{1}{2}{v_0}^2 +
G(\Xi^0) \,d\rho \; < \; \infty.
\end{equation*}
Moreover the following holds: $(\beta-\beta_0)/h = 0 = 3v'$,
$(\alpha-\alpha_0)/h = 0 = 3{\alpha_0}^{2/3}v$, and
$(\gamma-\gamma_0)/h = 0 = 2{\alpha_0}^{1/3}v$. Hence
$(\alpha,\beta,\gamma,v)\in \mathcal{A_{\lambda}}$.
\end{proof}
\par\medskip

\begin{lmm}[\bf I-bounded sequences]\label{IBNDSEQLMM}
Let $\BBF{(\alpha_n,\beta_n,\gamma_n,v_n)}_{n \in \mathbb{N}}
\subset \mathcal{A}_{\lambda}$ and
\begin{equation}\label{SUPI}
M=\sup_{n\in \mathbb{N}} \, I(\alpha_n,\beta_n,\gamma_n,v_n) \,
<\,\infty.
\end{equation}
Then $\exists$ $(\alpha,\beta,\gamma,v)\in X$ and a subsequence
$\BBF{(\alpha_{\mu},\beta_{\mu},\gamma_{\mu},v_{\mu})}$ s.t.
\begin{equation}\label{WKCONV}
\begin{aligned}
\alpha_{\mu}&\rightharpoonup\alpha                        &in&\;\;W^{1,1}, \;\;&\alpha_{\mu}/\rho        &\rightharpoonup  \alpha/\rho     &in& \;\;L^p,\\
\gamma_{\mu}/\rho^{2/3}&\rightharpoonup\gamma/\rho^{2/3}  &in&\;\;L^q,     \;\;&\gamma_{\mu}'\rho^{1/3}  &\rightharpoonup\gamma'\rho^{1/3} &in& \;\;L^q,\\
v_{\mu}                   &\rightharpoonup  v             &in&\;\;L^2,     \;\;&v_{\mu}'\rho^{2/3}       &\rightharpoonup  v'\rho^{2/3}    &in& \;\;L^{3p},\\
\beta_{\mu}\rho^{2/3}&\rightharpoonup\beta\rho^{2/3}      &in&\;\;L^{3p}.&&\\
\end{aligned}
\end{equation}

\end{lmm}

\begin{proof}
First, $\alpha_n \geqslant 0$, $\alpha_n'>0$ a.e. and
$\alpha_n(1)=\lambda$ imply that $|\alpha_n| \leqslant \lambda$.
Second, from (\ref{SUPI}) it follows $\int^1_0 \, h(\alpha_n')\,
d\rho<M, \; \forall n$. By the de la Vall\'{e}e Poussin
criterion there exists $\alpha \in W^{1,1}$ and a subsequence
$\{\alpha_{s}\}$ such that
$\alpha_{s}
\rightharpoonup \alpha$ weakly in $W^{1,1}$.\par\medskip

By (A3) there exist constants $C_1,C_2$ s.t. $\varphi(x)
\geqslant C_1|x|^{3p}-C_2$, $\psi(x) \geqslant C_1|x|^p-C_2$ and
$g(x) \geqslant C_1|x|^q-C_2$, and thus
\begin{equation}\label{COERSIVITY}
\begin{aligned}
M &\geqslant I(\alpha_{s},\beta_{s},\gamma_{s},v_{s})
\geqslant \int\limits_0^1 \, \f{1}{2}(v_{s}-v_0)^2 \, d\rho\\
& + C_1 \int\limits_0^1 \, |\beta_{s}\rho^{2/3}|^{3p} +
2|\alpha_{s}/\rho|^p + |\gamma_{s}/\rho^{2/3}|^q +
\f{3}{2}|\gamma_{s}'\rho^{1/3}|^q \,d\rho - 4C_2
\end{aligned}
\end{equation}
This implies for $1<p,q<\infty$ that
$\alpha/\rho \in L^p$ and there exist $\beta\in
X_2,\gamma \in X_3$, and $v \in L^2$ and a subsequence
$\{\alpha_{\mu},\beta_{\mu},\gamma_{\mu},v_{\mu}\}$ of
$\{\alpha_{s},\beta_{s},\gamma_{s},v_{s}\}$ such that
(\ref{WKCONV})$_{2,3,4,5,6}$ hold. \par\bigskip

Finally, as $(\alpha_{\mu},\beta_{\mu},\gamma_{\mu},v_{\mu}) \in
\mathrm{A}_{\lambda}$ we have
$3v_{\mu}'\rho^{2/3}=(\beta_{\mu}-\beta_0)\rho^{2/3}/h$. Then by
(\ref{WKCONV})$_3$ we get $3v_{\mu}'\rho^{2/3} \rightharpoonup
(\beta-\beta_0)\rho^{2/3}/h$ in $L^{3p}$. Then by
(\ref{WKCONV})$_6$ for each $f\in C_0^{\infty}(0,1)$
\begin{equation}
\begin{aligned}
\int_0^1 \, vf' \, d\rho &= \lim_{\mu\to\infty}\int_0^1 \, v_{\mu}f' \, d\rho\\
 &=-\lim_{\mu\to\infty}\int_0^1 \, v_{\mu}'f \, d\rho = -\int_0^1 \, \f{1}{3h}(\beta-\beta_0)f \,
d\rho
\end{aligned}
\end{equation}
and hence $v'=(\beta-\beta_0)/3h$. Therefore $v\in Y$ and
$v_{\mu}'\rho^{2/3} \rightharpoonup v'\rho^{2/3}$.
\end{proof} \par\medskip

\begin{thm}[\bf Lower semi-continuity]\label{LWSMCONTTHM}
Let $\{ (\alpha_n,\beta_n,\gamma_n,v_n) \}_{n \in
\mathbb{N}}\subset \mathcal{A}_{\lambda}$,
$(\alpha,\beta,\gamma,v)\in X$ satisfy (\ref{SUPI}) and
(\ref{WKCONV}). Then $(\alpha,\beta,\gamma,v)\in
\mathcal{A}_{\lambda}$ and
\begin{equation}\label{LWSMCT}
I(\alpha,\beta,\gamma,v) \leqslant
\liminf_{n\to\infty}I(\alpha_n,\beta_n,\gamma_n,v_n)=s<\infty.
\end{equation}

\end{thm}

\begin{proof}
By hypothesis $0 \leqslant
I_n=I(\alpha_n,\beta_n,\gamma_n,v_n) \leqslant M$, $\forall n \in
\mathbb{N}$ and thus $s<\infty$. Recall that $\alpha_n \rightharpoonup \alpha$
weakly in $W^{1,1}$ and (along a subsequence) uniformly on $C[0,1]$. Since
$\alpha_n(1) = \lambda$ we obtain $\alpha (1) =\lambda$.
Moreover,
\begin{equation}\label{WKCVAPLZ} \lim_{n\to\infty} \int_0^1 \,
\alpha_n'\chi_{ \BBF{\alpha'<0} } \, d\rho = \int_0^1 \,
\alpha'\chi_{ \BBF{\alpha'<0} } \, d\rho.
\end{equation}
Since $\alpha_n'>0$ a.e. we obtain $\int_0^1 \,
\alpha'\chi_{ \BBF{\alpha'<0} } \, d\rho \geqslant 0$, and thus $m\BBF{\alpha'<0}=0$.
\par\medskip

Now, denote $A=\BBF{\rho\in(0,1):\,\alpha'=0}$ and show that
$m(A)=0$. We will argue by contradiction. Assume that
$m(A)=\varepsilon>0$. Then (\ref{WKCONV}) implies
\begin{equation}\label{WKCVAPEZ}
\lim_{n\to\infty} \int_0^1 \, \alpha_n'\chi_{A} \, d\rho =
\int_0^1 \, \alpha'\chi_{A} \, d\rho=0.
\end{equation}
Then, as $\alpha_n'>0$ a.e.,
$\lim_{n\to\infty}\int_0^1\,|\alpha_n'\chi_{A}|\,d\rho=0$. Hence
$\alpha_n'\chi_{A} \to 0$ in $L^1$. We extract a subsequence
$\BBF{\alpha_{n_k}'}$ such that $\alpha_{n_k}'\chi_{A}\to 0$ a.e.
 $\rho\in(0,1)$. Now, by Egoroff's theorem there exists a measurable set
$B \subset A$ such that $m\BBR{B}>\varepsilon/2$ and
$\alpha_{n_k}'\to 0$ uniformly on $B$.
Next, observe that
$$
\int_0^1 h (\alpha_{n_k}^\prime) d\rho \ge \int_{B} h
(\alpha_{n_k}^\prime) d\rho \ge m (B) \left ( \inf_{\rho \in B} h
(\alpha_{n_k}^\prime) \right ) =: m (B) \, \mu_{n_k}
$$
Since $ \mu_{n_k} \to \infty$ this contradicts \eqref{SUPI}. We conclude that $m(A)=0$.
\par\medskip

Next we prove $\alpha\geqslant 0$ a.e. $\rho\in(0,1)$. Again \eqref{WKCONV}$_1$ implies
\begin{equation}\label{WKCVALZ}
\lim_{n\to\infty} \int_0^1 \, \alpha_n\chi_{\BBF{\alpha<0}} \,
d\rho = \int_0^1 \, \alpha\chi_{\BBF{\alpha<0}} \, d\rho \; \ge  \; 0 \, ,
\end{equation}
and thus  $m\BBF{\alpha<0}=0$.
This concludes that $\alpha$ satisfies all restrictions of membership in
$\mathcal{A}_{\lambda}$.
\par\medskip

Next, by (A2) we get
\begin{equation}\label{CVXINQ1}
\begin{aligned}
\varphi(\beta_n\rho^{2/3})    \, &\geqslant \, \varphi(\beta\rho^{2/3})    + \varphi'(\beta\rho^{2/3})(\beta_n-\beta)\rho^{2/3},\\
\psi\BBR{{\alpha_n}/{\rho}}   \, &\geqslant \, \psi\BBR{{\alpha}/{\rho}}   + \psi'\BBR{{\alpha}/{\rho}}(\alpha_n-\alpha)/\rho,\\
g\BBR{\gamma_n/\rho^{2/3}}    \, &\geqslant \, g\BBR{\gamma/\rho^{2/3}}    + g'\BBR{\gamma/\rho^{2/3}}(\gamma_n-\gamma)/\rho^{2/3},\\
g\BBR{3\gamma_n'\rho^{1/3}/2} \, &\geqslant \, g\BBR{3\gamma'\rho^{1/3}/2} + g'\BBR{3\gamma'\rho^{1/3}/2}(\gamma_n'-\gamma)3\rho^{1/3}/2 \\
\end{aligned}
\end{equation}
a.e. $\rho\in(0,1)$. As
$(\alpha,\beta,\gamma,v),(\alpha_n,\beta_n,\gamma_n,v_n)\in X$,
from (A3) it follows that the right hand side of each of the
inequalities in (\ref{CVXINQ1}) are integrable and
\begin{equation}\label{DERSPACES}
\begin{aligned}
&\varphi'(\beta\rho^{2/3})    \in  L^{\f{3p}{3p-1}},\;\;\;
\psi'\BBR{{\alpha}/{\rho}}   \in  L^{\f{p}{p-1}},\\
&\tx{and} \;\;\;
g'\BBR{\gamma/\rho^{2/3}},g'\BBR{3\gamma_n'\rho^{1/3}/2} \in
L^{\f{q}{q-1}}.\\
\end{aligned}
\end{equation}\par\medskip

Take an arbitrary $0<\delta<1$ and set
$A_{\delta}=\BBF{\rho\in(0,1): \, \delta\leqslant \alpha'
\leqslant 1/\delta}$. Then by (A2)
\begin{equation}\label{CVXINQ2}
h(\alpha_n') \geqslant h(\alpha')\chi_{A_{\delta}} +
h'(\alpha')(\alpha_n'-\alpha')\chi_{A_{\delta}} \;\; \tx{a.e.
}\;\;\rho\in(0,1).
\end{equation}
Moreover, (A1) and (A2) together imply
\begin{equation*}
\begin{aligned}
0 \; \leqslant \; h(\alpha')\chi_{A_{\delta}}
&+|h'(\alpha')|   \chi_{A_{\delta}} \\
&\leqslant \;
2\max(h(\delta),h(1/\delta),|h'(\delta)|,|h'(1/\delta)|).
\end{aligned}
\end{equation*}
Hence
\begin{equation}\label{HHPSPACE}
h(\alpha')\chi_{A_{\delta}},\; h'(\alpha')\chi_{A_{\delta}} \in
L^{\infty} \, ,
\end{equation}
and we conclude that the right hand side of (\ref{CVXINQ2}) is
integrable.\par\medskip

Finally,
\begin{equation}\label{CVXINQ3}
(v_n-v_0)^2 \geqslant (v-v_0)^2 + 2(v-v_0)(v_n-v) \;\;\tx{a.e.}
\;\rho\in(0,1),
\end{equation}
where right hand side is integrable as $v,v_n,v_0 \in L^2$.
\par\medskip

Following the discussion above, (\ref{CVXINQ1})-(\ref{CVXINQ3})
imply
\begin{eqnarray*}
\begin{aligned}
I_n \geqslant \int_{0}^{1}& \f{1}{2}(v-v_0)^2
+\varphi(\beta\rho^{2/3})
+2\psi\BBR{{\alpha}/{\rho}} \\
&\txq{}+g\BBR{\gamma/\rho^{2/3}}+2g\BBR{3\gamma'\rho^{1/3}/2} \,
d\rho +\int_0^1 \, h(\alpha')\chi_{A_{\delta}}\,d\rho\\
+\int_0^1 & \,(v-v_0)(v_n-v)
+\varphi'(\beta\rho^{2/3})(\beta_n-\beta)\rho^{2/3}\\
&+2\psi'\BBR{{\alpha}/{\rho}}(\alpha_n-\alpha)/\rho
+g'\BBR{\gamma/\rho^{2/3}}(\gamma_n-\gamma)/\rho^{2/3}\\
&\txq{}+g'\BBR{3\gamma'\rho^{1/3}/2}(\gamma_n'-\gamma)3\rho^{1/3}
+h'(\alpha')\chi_{A_{\delta}}(\alpha_n'-\alpha')
 \,d\rho\\
= J &+ J_{\delta} + J_n .
\end{aligned}
\end{eqnarray*}
Then, letting $n\to \infty$, we obtain
\begin{equation*}
\infty \,>\,s\,=\,\liminf_{n\to\infty}I_n \,\geqslant\, J +
J_{\delta} + \liminf_{n\to\infty}J_n.
\end{equation*}
Now from (\ref{WKCONV}), (\ref{DERSPACES}), (\ref{HHPSPACE}),  and
$v-v_0\in L^2$ it follows that $\lim_{n\to\infty}J_n=0$ and hence
\begin{equation}\label{LWSMCTINT}
\infty \,>\,s\,=\,\liminf_{n\to\infty}I_n \,\geqslant\, J +
\int_0^1 \, h(\alpha')\chi_{A_\delta} \,d\rho.
\end{equation}
Now, as $\alpha'>0$ a.e. $\rho\in(0,1)$ and $\alpha'\in L^1$, the
set $\BBF{\alpha'=0}\bigcup\BBF{\alpha'=\infty}$ is of measure
zero and hence
\begin{equation}\label{HDELTALIM}
\lim_{\delta\to
0+}h(\alpha')\chi_{A_{\delta}}=h(\alpha')\chi_{\{0<\alpha'<\infty
\}}=h(\alpha')\;\;\tx{a.e.}\;\rho\in(0,1).
\end{equation}
Finally, let $\delta\to 0+$. Then from (\ref{LWSMCTINT}),
(\ref{HDELTALIM}) and Monotone Convergence Theorem it follows
\begin{equation*}
\infty \,>\,s\,=\,\liminf_{n\to\infty}I_n \,\geqslant\, J +
\int_0^1 \, h(\alpha')\,d\rho \,=\,I(\alpha,\beta,\gamma,v)
\end{equation*}
and hence (\ref{LWSMCT}) holds.
Since $(\alpha_n,\beta_n,\gamma_n,v_n)\in \mathcal{A}_{\lambda}$,
and the other constraints are linear, one easily checks that the limiting
$(\alpha,\beta,\gamma,v)\in \mathcal{A}_{\lambda}$.
\end{proof}\par\medskip

\begin{thm}[\bf Existence]\label{EXISTENCE}
There exists $(\alpha,\beta,\gamma,v) \in \mathcal{A}_{\lambda}$
satisfying
\begin{equation}
I(\alpha,\beta,\gamma,v)=\inf_{\mathcal{A}_{\lambda}}I(\bar{\alpha},\bar{\beta},\bar{\gamma},\bar{v}).
\end{equation}
\end{thm}
\begin{proof}
As $\mathcal{A}_{\lambda}$ is nonempty, we can set
$s=\inf_{\mathcal{A}_{\lambda}}I(\bar{\alpha},\bar{\beta},\bar{\gamma},\bar{v})$.
Then by definition of $\mathcal{A}_{\lambda}$ we have
$I(\bar{\alpha},\bar{\beta},\bar{\gamma},\bar{v})<\infty$ for each
$ (\bar{\alpha},\bar{\beta},\bar{\gamma},\bar{v})\in
\mathcal{A}_{\lambda}$. This implies that $s$ is finite.\par

Next, by definition of $s$ there exists $
\BBF{(\alpha_n,\beta_n,\gamma_n,v_n)}_{n\in\mathbb{N}}\in
\mathcal{A}_{\lambda}$ such that $\lim_{n\to\infty}I_n=s$ with
$I_n=I(\alpha_n,\beta_n,\gamma_n,v_n)$. Then, as
$\BBF{I_n}_{n\in\mathbb{N}}$ is bounded, lemma \ref{IBNDSEQLMM}
and Theorem \ref{LWSMCONTTHM} imply that $\exists
(\alpha,\beta,\gamma,v)\in \mathcal{A}_{\lambda}$ satisfying
$I(\alpha,\beta,\gamma,v)\leqslant \liminf_{n\to\infty}I_n=s$. In
this case the definition of $s$ implies
$I(\alpha,\beta,\gamma,v)=s$.
\end{proof}\par\medskip

\begin{thm}[\bf Uniqueness]\label{UNIQUENESS}
The minimizer $(\alpha,\beta,\gamma,v)\in \mathcal{A}_{\lambda}$ of
$I$ over $\mathcal{A}_{\lambda}$ is unique.
\end{thm}
\begin{proof}
We will argue by contradiction. Assume
$(\alpha,\beta,\gamma,v),(\bar{\alpha},
\bar{\beta},\bar{\gamma},\bar{v})\in \mathcal{A}_{\lambda}$
are two distinct minimizers. Then we consider
$(\f{\alpha+\bar{\alpha}}{2},\f{\beta+\bar{\beta}}{2},\f{\gamma+\bar{\gamma}}{2},
\f{v+\bar{v}}{2})$ and notice that it also belongs to $\mathcal{A}_{\lambda}$.
\par\smallskip

Define $A=\BBF{\rho\in(0,1): \, \alpha' \neq
\bar{\alpha}'}$. Then $mA>0$. Indeed, if $\alpha'=\bar{\alpha}'$
a.e., then $\alpha(1)=\bar{\alpha}(1)=\lambda$ implies
$\alpha=\bar{\alpha}$. In turn, this implies $v=\bar v'$, $\beta=\bar\beta$ and $\gamma = \bar \gamma$,
which  contradicts to the assumption that $(\alpha,\beta,\gamma,v)$ and
$(\bar{\alpha},\bar{\beta},\bar{\gamma},\bar{v})$ are
distinct.\par\medskip

Now, as $h''>0$, we have
\begin{equation*}
\f{h(\alpha') + h(\bar{\alpha}')}{2} \, > \,
h\BBR{\f{\alpha'+\bar{\alpha}'}{2}}, \;\;\; \rho\in A \, ,
\end{equation*}
and thus, as $mA$ is positive,
\begin{equation*}
\int_0^1 \, \f{h(\alpha') + h(\bar{\alpha}')}{2} \, d\rho \, >
\,\int_0^1 h\BBR{\f{\alpha'+\bar{\alpha}'}{2}} \, d\rho \, .
\end{equation*} \par\medskip

Let
$s=\inf_{\mathcal{A}_{\lambda}}I(\tilde{\alpha},\tilde{\beta},\tilde{\gamma},\tilde{v})$.
Then by the inequality above and convexity of $\varphi,\psi$ and
$g$ we obtain
\begin{equation}\label{THDMIN}
s=\f{I(\alpha,\beta,\gamma,v)+I(\bar{\alpha},\bar{\beta},\bar{\gamma},\bar{v})}{2}
\,>\,
I\BBR{\f{\alpha+\bar{\alpha}}{2},\f{\beta+\bar{\beta}}{2},\f{\gamma+\bar{\gamma}}{2},\f{v+\bar{v}}{2}} \, ,
\end{equation}
which, since
$\BBR{\f{\alpha+\bar{\alpha}}{2},\f{\beta+\bar{\beta}}{2},\f{\gamma+\bar{\gamma}}{2},\f{v+\bar{v}}{2}}\in
\mathcal{A}_{\lambda}$, contradicts the definition of $s$. Hence
$(\alpha,\beta,\gamma,v)=(\bar{\alpha},\bar{\beta},\bar{\gamma},\bar{v})$.
\end{proof}


\section{Euler-Lagrange Equations}
\label{seceullag}

Next, we show that the minimizer of $I$
satisfies the system (\ref{EXTDISCSYS}) a.e. $\rho\in(0,1)$. To
this end, in addition to (\ref{IDATAPROP}), we assume that the initial
iterate $(\alpha_0,\beta_0,\gamma_0,v_0)$ satisfies for each $\delta\in(0,1)$
\begin{equation}\label{IDATAPROP2}
{\alpha_0'} \in L^{3p}(\delta,1) \bigcap L^q(\delta,1).
\end{equation}\par\medskip


\begin{thm}[\bf Weak Form]\label{ELWEAK}
Let  $(\alpha,\beta,\gamma,v)\in \mathcal{A}_{\lambda}$ be the
minimizer of $I$ over $\mathcal{A}_{\lambda}$ and the initial iterate
$(\alpha_0,\beta_0,\gamma_0,v_0)$ satisfy (\ref{IDATAPROP}) and
(\ref{IDATAPROP2}). Let also
\begin{equation}\label{G1DEF}
G_1(\rho)=G_{,i}(\Xi) \Omega^i_{,1}(\Gamma^0)
\end{equation}
and
\begin{equation}\label{G2DEF}
G_2(\rho) =G_{,i}(\Xi) \BBR{\Omega^i_{,2}(\Gamma^0)
 +\Omega^i_{,3}(\Gamma^0)}
\end{equation}
Then, for each $\delta\in(0,1)$,
\begin{equation*}
\rho^{2/3}G_1(\rho)\in W^{1,1}(\delta,1) \, , \quad
\rho^{-1/3}G_2(\rho)\in L^1(\delta,1) \, ,
\end{equation*}
and for a.e. $\rho\in(0,1)$
\begin{equation}\label{G1ABSCONT}
3\rho^{2/3} G_1(\rho)=\int_1^{\rho}  \BBR{s^{-1/3}G_2(s) +
\f{v(s)-v_0(s)}{h} } ds + \tx{const.}
\end{equation}
Moreover, for each $\delta\in(0,1)$,
\begin{equation}\label{MINPROP2}
{\alpha}'\in L^{3p}(\delta,1) \bigcap L^q(\delta,1).
\end{equation}
\end{thm}

\begin{proof}
Fix $k\in\mathbb{N}$ and define $S_k=\{ \rho\in[1/k,1): \;
1/k<\alpha'<k \}$. Let $f\in L^{\infty}$ with
$\int_{S_k}\,f\, d\rho=0$. We denote by $\chi_k=\chi_{_{S_k}}$,
$l_k=\alpha_0(1/k)$ and set
\begin{equation}\label{MUDEF}
\mu(\rho)=\int_0^{\rho}\, \chi_k(s)f(s)\,ds.
\end{equation} \par\medskip

Before proceeding further we make the following remark. Let
$t\in\RR$ and $F(x)=x^t,x\in\RR_{+}$. Take $\delta\in(0,1)$. Then,
as $\alpha_0 \in W^{1,1}$, $\alpha_0 \geqslant 0$ and ${\alpha_0'}>0$
a.e. $\rho\in(0,1)$ we must have
$0<\alpha_0(\delta)\leqslant\alpha_0\leqslant\lambda$ for all
$\rho\in(\delta,1)$. Hence $|F'(\alpha_0)|\leqslant
t\BBR{\alpha_0(\delta)+\lambda}^{t-1}$ for all
$\rho\in(\delta,1)$.  Therefore we conclude that for each
$t\in\RR$ and $\delta\in(0,1)$
\begin{equation}\label{ALPHAZTDER}
{\alpha_0}^t \in W^{1,1}(\delta,1) \;\;\;\tx{with}\;\;\;
\f{d}{d\rho}\BBR{{\alpha_0}^t}=t{\alpha_0}^{t-1}{\alpha_0'}.
\end{equation}\par\medskip

(i) {\it Step 1. Definition of the variation}.
For $|\eps|<\f{1}{6k(\|f\|_{\infty}+1)}\,$
we define $(\alpha_{\eps},\beta_{\eps},\gamma_{\eps},v_{\eps})$ by
\begin{equation}\label{EPSELTDEF}
\begin{aligned}
v_{\eps} &= v + \eps\f{\mu}{h{\alpha_0}^{2/3}}\\
\alpha_{\eps} &= \alpha_0 + h\BBR{ 3v_{\eps}{\alpha_0}^{2/3} } =
\alpha +
3\eps\mu\\
\beta_{\eps} &= \beta_0 + h\BBR{ 3v_{\eps}'} = \beta + 3\eps
\BBR{\f{\mu}{{\alpha_0}^{2/3}}}'\\
\gamma_{\eps} &= \gamma_0 + h \BBR{2v_{\eps}{\alpha_0}^{1/3}} =
\gamma + 2 \eps \f{\mu}{{\alpha_0}^{1/3}} \, .
\end{aligned}
\end{equation}
Due to (\ref{ALPHAZTDER}), $(\alpha_{\eps},\beta_{\eps},\gamma_{\eps},v_{\eps})$ is
well-defined. We next prove:

\medskip
\begin{lmm}
The variation $(\alpha_{\eps},\beta_{\eps},\gamma_{\eps},v_{\eps})\in
\mathcal{A}_{\lambda}$.
\end{lmm}

\begin{proof}
 First, we notice that
\begin{equation}\label{EQLONINT}
(\alpha_{\eps},\beta_{\eps},\gamma_{\eps},v_{\eps}) =
(\alpha,\beta,\gamma,v) \;\;\; \tx{if} \;\;\; \rho\in (0,1/k).
\end{equation}
Then we check that
\begin{equation*}
\alpha_{\eps}(1) = \alpha(1) + 3 \eps \int_{S_k} \,
f(s)\,ds=\lambda.
\end{equation*}
Next, we see that $\alpha_{\eps}'=\alpha' + 3\eps \chi_k f$ and
therefore
\begin{equation}\label{AEPSPPROP}
\begin{aligned}
\alpha_{\eps}' = \alpha', \;\;\; \rho\notin S_k,\\
\f{1}{2k} \leqslant \alpha_{\eps}' \leqslant k+1,  \;\;\;\rho\in S_k.\\
\end{aligned}
\end{equation}
This implies that $\alpha_{\eps}>0$ a.e. $\rho\in(0,1)$ and hence
(\ref{EQLONINT}) implies $\alpha_{\eps}\geqslant0$.\par\medskip

Now we make the following estimates. First, we see that
\begin{equation*}\label{MUESTIM}
|\mu'|+\BBA{\f{\mu}{\rho}} +
\BBA{\f{\mu}{\rho^{2/3}{\alpha_0}^{1/3}}} +
\BBA{\f{\mu}{h{\alpha_0}^{2/3}}} \leqslant
\BBN{f}_{\infty}\BBR{1+k
 + \f{k^{2/3}}{l_k^{1/3}} + \f{1}{h l_k^{2/3}}}
\end{equation*}
and for $j=1,2$
\begin{equation*}\label{MUPESTIM}
\BBA{\BBR{\f{\mu}{{\alpha_0}^{j/3}}}'} \leqslant \BBN{f}_{\infty}
\BBR{ l_k^{-j/3} + l_k^{-(1+j/3)}\BBA{{\alpha_0'}}}.
\end{equation*}
Thus we conclude that there exists $C$ such that $\forall \rho \in
(1/k,1)$
\begin{equation}\label{EPSELTEST1}
|\alpha_{\eps}'- \alpha' | +|\alpha_{\eps}/\rho - \alpha/\rho |
+|\gamma_{\eps}/\rho^{2/3}- \gamma/\rho^{2/3} |+|v - v_{\eps}|
\leqslant \eps C
\end{equation}
and
\begin{equation}\label{EPSELTEST2}
|\beta_{\eps} \rho^{2/3}- \beta\rho^{2/3} |
+|\gamma_{\eps}'\rho^{1/3}- \gamma'\rho^{1/3} | \leqslant \eps C
\BBR{ 1 + |{\alpha_0'}|}.
\end{equation}
As $(\alpha,\beta,\gamma,v)\in X$, the last two inequalities imply
$(\alpha_{\eps},\beta_{\eps},\gamma_{\eps},v_{\eps})\in
X$.\par\medskip

Further, by (A3), (\ref{EPSELTEST1}) and (\ref{EPSELTEST2}) we
conclude that there exists $C$ such that for all $\rho\in (1/k,1)$
\begin{equation*}
\begin{aligned}
\psi(\alpha_{\eps}/\rho)         &\,\leqslant\, C\BBR{|\alpha/\rho|^p+1}\\
\varphi(\beta_{\eps}\rho^{2/3})  &\,\leqslant\, C\BBR{|\beta\rho^{2/3}|^{3p}+|{{\alpha_0'}}|^{3p}+1} \\
g(\gamma_{\eps}/\rho^{2/3})      &\,\leqslant\, C\BBR{|\gamma/\rho^{2/3}|^q+1}\\
g(3\gamma_{\eps}'\rho^{1/3}/2)   &\,\leqslant\, C\BBR{|\gamma'\rho^{1/3}|^q + |{{\alpha_0'}}|^{q}+1}. \\
\end{aligned}
\end{equation*}
By (\ref{AEPSPPROP}) we also have
\begin{equation}\label{HEPSPPROP}
\begin{aligned}
h(\alpha_{\eps}') &= h(\alpha'), \;\;\; \rho\notin {S_k},\\
h(\alpha_{\eps}') &\leqslant \max_{\f{1}{2k} \leqslant \delta \leqslant k}|{h(\delta)}|=M_k,  \;\;\;\rho\in S_k\\
\end{aligned}
\end{equation}
and hence
\begin{equation}
h(\alpha_{\eps}') \leqslant h(\alpha') + M_k, \;\;\;\rho \in
(0,1).
\end{equation} \par\medskip

Now, similarly to (\ref{XIDEFDISC}), set
\begin{equation}
\Xi_{\eps}= \;\BBR{\beta_{\eps} \rho^{2/3},
\f{\alpha_{\eps}}{\rho},\f{\alpha_{\eps}}{\rho}, \f{\gamma_{\eps}}
{\rho^{1/3}}, \f{3\gamma_{\eps}'}{2}\rho^{2/3},
\f{3\gamma_{\eps}'}{2}\rho^{2/3},\alpha_{\eps}'\rho^{2/3}}.
\end{equation}
Then, by the discussion above, it follows that
\begin{equation}\label{GEPSPPROP1}
G(\Xi_{\eps})+\f{(v_{\eps}-v_0)^2}{2}= G(\Xi)+\f{(v-v_0)^2}{2},
\;\;\; \rho\in (0,1/k) \, ,
\end{equation}
and there exists $C$ such that for $\rho\in (1/k,1)$
\begin{equation}\label{GEPSPPROP2}
\begin{aligned}
G(\Xi_{\eps})+\f{(v_{\eps}-v_0)^2}{2} \leqslant C&\left(1+|\beta\rho^{2/3}|^{3p}+|{{\alpha_0'}}|^{3p}+|\alpha/\rho|^p +|\gamma/\rho^{2/3}|^q\right. \\
&\left.+|\gamma'\rho^{1/3}|^q+ |{{\alpha_0'}}|^{q} + |v|^2 + |v_0|^2+h(\alpha')\right).\\
\end{aligned}
\end{equation}
As $I(\alpha,\beta,\gamma,v)<\infty$, (\ref{GEPSPPROP1}) and
(\ref{GEPSPPROP2}) imply
$I(\alpha_{\eps},\beta_{\eps},\gamma_{\eps},v_{\eps})<\infty$ and
hence by construction and the above discussion we get
$(\alpha_{\eps},\beta_{\eps},\gamma_{\eps},v_{\eps})\in
\mathcal{A}_{\lambda}$.
\end{proof} \par\medskip

\medskip
{\it Step 2}. The next objective is to validate the formal identity
\begin{equation}
\label{IDERIV}
\begin{aligned}
\left.\f{d}{d\eps}I(\alpha_{\eps},\beta_{\eps},\gamma_{\eps},v_{\eps})\right|_{\eps=0}=\int_0^1
 \left. \f{d}{d\eps}\BBR{\f{(v_{\eps}-v_0)^2}{2}+G(\Xi_{\eps})}
\right|_{\eps=0} \, d\rho=0.
\end{aligned}
\end{equation}\\
This will require several detailed estimations presented below.

\medskip
At this point, let us make estimates of the following difference
quotients on the interval  $\rho\in (1/k,1)$. First, by
(\ref{EPSELTEST1}) we get
\begin{equation}\label{VDQ}
\begin{aligned}
\f{1}{\eps}|(v_{\eps}-v_0)^2-(v-v_0)^2| =& \f{1}{\eps}
|v_{\eps}-v||v_{\eps}+v-2v_0|\\
& \leqslant  C\BBR{|v|+|v_0|+1}.
\end{aligned}
\end{equation}
Further, by the Mean Value Theorem
\begin{equation*}
\f{1}{\eps}|\varphi(\beta_{\eps}\rho^{2/3})-
\varphi(\beta\rho^{2/3})|=\f{1}{\eps}|\varphi'(\tau_{\eps})|
|\beta_{\eps}\rho^{2/3}- \beta\rho^{2/3}|,
\end{equation*}
where $\min(\beta,\beta_{\eps})\rho^{2/3} \leqslant \tau_{\eps}
\leqslant \max(\beta,\beta_{\eps})\rho^{2/3}$. Hence from
(\ref{EPSELTEST2}) it follows $|\tau_{\eps}| \leqslant
|\beta\rho^{2/3}|+\eps C (|{\alpha_0'}|+1)$ and therefore (A4)
implies
\begin{equation*}
|\varphi'(\tau_{\eps})| \leqslant
C\BBR{|\beta\rho^{2/3}|^{3p-1}+|{\alpha_0'}|^{3p-1}+1}.
\end{equation*}
Thus
\begin{equation}\label{PHIBETADQ}
\begin{aligned}
\f{1}{\eps}|\varphi(\beta_{\eps}\rho^{2/3})-
&\varphi(\beta\rho^{2/3})|\\
&\leqslant
C\BBR{|\beta\rho^{2/3}|^{3p-1}+|{\alpha_0'}|^{3p-1}+1}\BBR{|{\alpha_0'}|+1}.
\end{aligned}
\end{equation}
Similarly,
\begin{equation*}
\f{1}{\eps}|\psi(\alpha_{\eps}/\rho) -
\psi(\alpha/\rho)|=\f{1}{\eps}|\psi'(\tau_{\eps})||\alpha_{\eps}/\rho-\alpha/\rho|,
\end{equation*}
where $\min(\alpha_{\eps},\alpha)/\rho \leqslant \tau_{\eps}
\leqslant \max(\alpha_{\eps},\alpha)/\rho$. Hence
$|\tau_{\eps}|\leqslant |\alpha/\rho|+\eps C$ and  (A4)
implies
\begin{equation*}
|\psi'(\tau_{\eps})| \leqslant C \BBR{(|\alpha/\rho|+1 )^{p-1}+1}
\end{equation*}
and hence
\begin{equation}\label{PSIALPHADQ}
\f{1}{\eps}|\psi(\alpha_{\eps}/\rho) - \psi(\alpha/\rho)|
\leqslant C \BBR{(|\alpha/\rho|+1 )^{p-1}+1}.
\end{equation}
Next,
\begin{equation*}
\f{1}{\eps} |g(\gamma_{\eps}/\rho^{2/3}) - g(\gamma/\rho^{2/3})|=
\f{1}{\eps}|g'(\tau_{\eps})||\gamma_{\eps}/\rho^{2/3} -
\gamma/\rho^{2/3}|,
\end{equation*}
where $\min(\gamma_{\eps},\gamma)/\rho^{2/3} \leqslant
|\tau_{\eps}| \leqslant \max(\gamma_{\eps},\gamma)/\rho^{2/3}$ and
hence $|\tau_{\eps}| \leqslant |\gamma/\rho^{2/3}|+\eps C$. Then
by (A4)
\begin{equation*}
|g'(\tau_{\eps})| \leqslant C \BBR{
(|\gamma/\rho^{2/3}|+1)^{q-1}+1}
\end{equation*}
and hence
\begin{equation}\label{GGAMMADQ}
\f{1}{\eps}|g(\gamma_{\eps}/\rho^{2/3}) - g(\gamma/\rho^{2/3})|
\leqslant C \BBR{(|\gamma/\rho^{2/3}|+1 )^{q-1}+1}.
\end{equation}
Further,
\begin{equation*}
\f{1}{\eps} |g(3\gamma_{\eps}'\rho^{1/3}/2) -
g(3\gamma'\rho^{1/3}/2)|=
\f{3}{2\eps}|g'(\tau_{\eps})||\gamma_{\eps}'\rho^{1/3} -
\gamma'\rho^{1/3}|,
\end{equation*}
where $\f{3}{2}\min(\gamma_{\eps}',\gamma')\rho^{1/3}\leqslant
|\tau_{\eps}| \leqslant
\f{3}{2}\max(\gamma_{\eps}',\gamma')\rho^{1/3}$. Hence we must
have $$|\tau_{\eps}| \leqslant
\f{3}{2}\BBR{|\gamma'\rho^{1/3}|+\eps C (|{\alpha_0'}|+1)}.$$ Then
(A4) implies
\begin{equation*}
|g'(\tau_{\eps})| \leqslant C \BBR{
(|\gamma'\rho^{1/3}|+|{\alpha_0'}|+1)^{q-1}+1}
\end{equation*}
and hence
\begin{equation}\label{GGAMMAPDQ}
\begin{aligned}
\f{1}{\eps}|g(3\gamma_{\eps}'\rho^{1/3}/2) - &g(3\gamma'\rho^{1/3}/2)|\\
&\leqslant C \BBR{
(|\gamma'\rho^{1/3}|+|{\alpha_0'}|+1)^{q-1}+1}\BBR{|{\alpha_0'}|+1}.
\end{aligned}
\end{equation}
Finally, if $\rho\notin S_k$, then
$\f{1}{\eps}|h(\alpha_{\eps}')-h(\alpha')|=0$ and if $\rho \in
S_k$, we get
\begin{equation*}
\f{1}{\eps}|h(\alpha_{\eps}')-h(\alpha')|=\f{1}{\eps}|h'(\tau_{\eps})||\alpha_{\eps}'-\alpha'|,
\end{equation*}
where $\min(\alpha_{\eps}',\alpha') \leqslant \tau_{\eps}
\leqslant  \max(\alpha_{\eps}',\alpha)$. Then by
(\ref{AEPSPPROP})$_2$ we get $\f{1}{2k}\leqslant\tau_{\eps}
\leqslant k+1$ and hence
\begin{equation*}
|h'(\tau_{\eps})| \leqslant \max_{\f{1}{2k}
\leqslant\delta\leqslant k+1}|h'(\delta)|.
\end{equation*}
Thus by (\ref{EPSELTEST1}) we conclude that for $\rho\in (1/k,1)$
\begin{equation}\label{HALPHAPDQ}
\f{1}{\eps}|h(\alpha_{\eps}')-h(\alpha')|\leqslant C.
\end{equation}\par\medskip
Thus (\ref{GEPSPPROP1}),(\ref{VDQ})-(\ref{HALPHAPDQ}) and the
assumptions on the initial iterate (\ref{IDATAPROP}) and
(\ref{IDATAPROP2}) imply that
\begin{equation*}
\begin{aligned}
&\f{1}{\eps} \BBA{ G(\Xi_{\eps}) + \f{(v_{\eps}-v_0)^2}{2} -
G(\Xi)-\f{(v-v_0)^2}{2} }
\end{aligned}
\end{equation*}
is bounded on $(0,1)$ by an integrable function.
Letting $\eps\to 0$, and using the Dominated Convergence theorem, (A2) and the
fact that $(\alpha,\beta,\gamma,v)$ is the minimizer, we obtain the identity \eqref{IDERIV}.

{\it Step 3. Conclusion of the computation}.
The last step is to compute the right hand side of \eqref{IDERIV}.
Note first that
\begin{equation*}
\begin{aligned}
\f{d\Xi^1_{\eps}}{d\eps} &= \f{d}{d\eps} \beta_{\eps}\rho^{2/3}                  = 3\BBR{\f{\mu}{{\alpha_0}^{2/3}}}'\rho^{2/3} \\
\f{d\Xi^2_{\eps}}{d\eps} &= \f{d\Xi^3_{\eps}}{d\eps} =\f{d}{d\eps} \BBR{\f{\alpha_{\eps}}{\rho}}           = \f{3\mu}{\rho}\\
\f{d\Xi^4_{\eps}}{d\eps} &= \f{d}{d\eps} \BBR{\f{\gamma_{\eps}}{\rho^{1/3}}}     = \f{2\mu}{{\alpha_0}^{1/3}\rho^{1/3}}  \\
\f{d\Xi^5_{\eps}}{d\eps} &= \f{d\Xi^6_{\eps}}{d\eps}=\f{d}{d\eps} \BBR{ \f{3}{2}\gamma_{\eps}'\rho^{2/3}} = 3\BBR{\f{\mu}{{\alpha_0}^{1/3}}}'\rho^{2/3} \\
\f{d\Xi^7_{\eps}}{d\eps} &= \f{d}{d\eps} \BBR{ \alpha_{\eps}'\rho^{2/3}}         = 3\mu'\rho^{2/3}\\
\end{aligned}
\end{equation*}
and
\begin{equation*}
\f{dv_{\eps}}{d\eps}=\f{\mu}{h{\alpha_0}^{2/3}}.
\end{equation*}\\
Then the integrand in (\ref{IDERIV}) is expressed by
\begin{equation*}
\begin{aligned}
\left.(v-v_0)\f{dv_{\eps}}{d\eps}\right|_{\eps=0}+G_{,i}(\Xi)\left.\f{d\Xi^i_{\eps}}{d\eps}\right|_{\eps=0}=a\mu+b\mu',
\end{aligned}
\end{equation*}
where
\begin{equation}\label{ALATDEF}
\begin{aligned}
a(\rho)&=-G_{,1}(\Xi) \f{2{\alpha_0'}}{{\alpha_0}^{5/3}} \rho^{2/3}+G_{,2}(\Xi)\f{3}{\rho}+G_{,3}(\Xi)\f{3}{\rho}\\
&+G_{,4}(\Xi)\f{2}{{\alpha_0}^{1/3}\rho^{1/3}}-\BBR{G_{,5}(\Xi)+G_{,6}(\Xi)}
\f{{\alpha_0'}}{{\alpha_0}^{4/3}}
\rho^{2/3}+\f{(v-v_0)}{h{\alpha_0}^{2/3}}
\end{aligned}
\end{equation}
and
\begin{equation}\label{BLATDEF}
\begin{aligned}
b(\rho)=\f{3\rho^{2/3}}{{\alpha_0}^{2/3}}\BBR{G_{,1}(\Xi)+G_{,5}(\Xi){\alpha_0}^{1/3}+G_{,6}(\Xi){\alpha_0}^{1/3}+G_{,7}(\Xi){\alpha_0}^{2/3}}.
\end{aligned}
\end{equation}
Thus by (\ref{IDERIV}) we have $(a\mu + b \mu')\in L^1$ and
\begin{equation}\label{IDERIVVIAAB}
\int_{1/k}^{1} \, (a\mu + b \mu') \, d\rho =0.
\end{equation} \par\smallskip
Now, we claim $a\in L^1(1/k,1)$. By (A3) and definition
(\ref{GPART}) of $G$ it follows that for $\rho\in (1/k,1)$

\begin{equation*}
\begin{aligned}
\BBA{G_{,1}(\Xi)\f{{\alpha_0'}}{{\alpha_0}^{5/3}} \rho^{2/3}}
\leqslant
\BBA{\f{{\varphi'(\beta\rho^{2/3})\alpha_0}'}{l_k^{5/3}}}
\leqslant C\BBR{ |\beta\rho^{2/3}|^{3p-1}+1}|{\alpha_0'}|,
\end{aligned}
\end{equation*}

\begin{equation*}
\begin{aligned}
\f{1}{\rho} \BBA{G_{,2}(\Xi)+G_{,3}(\Xi) }\leqslant
2k\BBA{\psi'(\alpha/\rho)} \leqslant C \BBR{|\alpha/\rho|^{p-1}+1
},
\end{aligned}
\end{equation*}
and
\begin{equation*}
\begin{aligned}
\BBA{G_{,4}(\Xi)\f{1}{{\alpha_0}^{1/3}\rho^{1/3}}} \leqslant
\f{k^{1/3}}{l_k^{1/3}} \BBA{g'(\gamma/\rho^{2/3})} \leqslant C
\BBR{ |\gamma/\rho^{2/3}|^{q-1}+1}.
\end{aligned}
\end{equation*}
As the right hand sides of the inequalities above are integrable
on $(1/k,1)$ we have $a\in L^1(1/k,1)$ and this, in turn,
implies $b \, \mu' \in L^1(1/k,1)$. Now, we set $z(\rho)=\int_1^{\rho}
\,a(s)\,ds$ for $\rho\in (1/k,1)$. Then $z$ is absolutely
continuous and so is $\mu z$. As $(\mu z)|_{\rho=1/k}=(\mu
z)|_{\rho=1}=0$ we get
\begin{equation*}
0=\int_{1/k}^{1} \, (\mu z)'\,d\rho= \int_{1/k}^1
\BBR{\mu'\int_1^{\rho} \,a(s)\,ds + \mu a}  \,d\rho.
\end{equation*}
Then (\ref{IDERIVVIAAB}) becomes
\begin{equation}\label{IDERIVVIAAB2}
\int_{S_k} \, \BBR{-\int_1^{\rho} \,a(s)\,ds + b }f \, d\rho =0.
\end{equation} \par\smallskip
By the properties of $f$ we obtain that for some constant
$c_k$
\begin{equation*}
b-\int_1^{\rho} \, a(s)\,ds=c_k \;\; a.e. \;\rho\in S_k.
\end{equation*}
As $k$ was arbitrary, the above equality is valid for all $k\in
\mathbb{N}$. In this case $S_k \subset S_{k+1}$ implies that
$c_k=c_{k+1}$. As $\bigcup_k S_k=\{\rho\in (0,1):\,
0<\alpha'<\infty\}$ and $m\BBR{(0,1)\backslash \bigcup_k S_k}=0$,
we conclude
\begin{equation}\label{ABSCONTB}
b-\int_1^{\rho} \, a(s)\,ds=\tx{const.} \;\; \tx{a.e.} \;\rho\in
(0,1).
\end{equation}
Now, let us fix $\delta\in(0,1)$. By the above argument $a\in L^1(\delta,1)$
and (\ref{ABSCONTB}) implies $b\in W^{1,1}(\delta,1)$ with
the weak derivative $b'=a$. Moreover, by (\ref{ALPHAZTDER}) we
have ${\alpha_0}^{2/3}\in W^{1,1}(\delta,1)$ and hence
$b{\alpha_0}^{2/3}\in W^{1,1}(\delta,1)$. At this point, we
compute
\begin{equation*}
D\Omega(\Gamma^0)=
\begin{bmatrix}
1& 0                              &0                              & 0                &{\alpha_0}^{1/3}                       & {\alpha_0}^{1/3}                    &{\alpha_0}^{2/3}   \\
0& 3\BBR{\f{\alpha_0}{\rho}}^{2/3}&0                              & {\alpha_0}^{1/3} &0                                      &\f{{\alpha_0'}\rho}{{\alpha_0}^{2/3}}&\f{{\alpha_0'}\rho}{{\alpha_0}^{1/3}}  \\
0& 0                              &3\BBR{\f{\alpha_0}{\rho}}^{2/3}& {\alpha_0}^{1/3} & \f{{\alpha_0'}\rho}{{\alpha_0}^{2/3}} &0                                    &\f{{\alpha_0'}\rho}{{\alpha_0}^{1/3}}   \\
\end{bmatrix}
\end{equation*}\\
and notice that definitions (\ref{ALATDEF}) and (\ref{BLATDEF})
of $a$ and $b$ imply
\begin{equation*}
\begin{aligned}
b{\alpha_0}^{2/3}=3\rho^{2/3}G_{,i}(\Xi)\Omega^i_{,1}(\Gamma^0)=3\rho^{2/3}G_1(\rho)
\end{aligned}
\end{equation*}
while its weak derivative is expressed as
\begin{equation*}
\begin{aligned}
\f{d}{d\rho}b{\alpha_0}^{2/3}&=a{\alpha_0}^{2/3}+b\f{2{\alpha_0'}}{3{\alpha_0}^{1/3}}=\\
&=\rho^{-1/3}G_{,i}(\Xi)\BBR{\Omega^i_{,2}(\Gamma^0)+\Omega^i_{,3}(\Gamma^0)}+\f{v-v_0}{h}\\
&=\rho^{-1/3}G_2(\rho)+\f{v-v_0}{h}.
\end{aligned}
\end{equation*}
We conclude that, for $\delta\in(0,1)$,
\begin{equation}
\rho^{2/3}G_1(\rho)\in W^{1,1}(\delta,1) \, , \quad
\rho^{-1/3}G_2(\rho)\in L^1(\delta,1) \, ,
\end{equation}
and for almost every $\rho\in(0,1)$
\begin{equation}
3\rho^{2/3}G_1(\rho)= \int_1^{\rho}
\BBR{s^{-1/3}G_2(s)+\f{v(s)-v_0(s)}{h} } ds + \tx{const}.
\end{equation}\par\medskip

Finally, to prove (\ref{MINPROP2}), we compute
\begin{equation*}
\begin{aligned}
\BBR{\alpha-{\alpha_0}}'&=h\BBR{3{\alpha_0}^{2/3}v}'=h\BBR{\f{2{\alpha_0'}}{{\alpha_0}^{1/3}}v+3{\alpha_0}^{2/3}v'}\\
&=(\alpha-\alpha_0)\f{2{\alpha_0'}}{3\alpha_0}+(\beta-\beta_0){{\alpha_0}^{2/3}}\\
\end{aligned}
\end{equation*}
and hence
\begin{equation}\label{ALPHAPEST1}
\begin{aligned}
\alpha'=\f{\alpha_0'}{3}\BBR{1+\f{2\alpha}{\alpha_0}}+(\beta-\beta_0){\alpha_0}^{2/3}.\\
\end{aligned}
\end{equation}
Similarly,
\begin{equation*}
\begin{aligned}
\BBR{\gamma-{\gamma_0}}'&=h\BBR{2{\alpha_0}^{1/3}v}'=\f{2}{3}\BBR{\f{\alpha-\alpha_0}{{\alpha_0}^{1/3}}}'\\
&=\f{2}{3{\alpha_0}^{1/3}}\BBR{\alpha'-\f{{\alpha_0'}}{3}\BBR{2+\f{\alpha}{\alpha_0}}}\\
\end{aligned}
\end{equation*}
and hence
\begin{equation}\label{ALPHAPEST2}
\begin{aligned}
\alpha'=\f{\alpha_0'}{3}\BBR{2+\f{\alpha}{\alpha_0}}+\f{3}{2}(\gamma'-{\gamma_0}'){\alpha_0}^{1/3}.
\end{aligned}
\end{equation}\\
Now, take $\delta\in(0,1)$. Then from (\ref{ALPHAPEST1}) and
(\ref{ALPHAPEST2}) it follows that for all $\rho\in(\delta,1)$
\begin{equation}
\begin{aligned}
|\alpha'| &\leqslant \f{|{\alpha_0'}|}{3}\BBR{1+
\f{2\lambda}{\alpha_0(\delta)}} +|\beta-\beta_0|
{\lambda}^{2/3}\\
\end{aligned}
\end{equation}
and
\begin{equation}
\begin{aligned}
|\alpha'| &\leqslant
\f{|{\alpha_0'}|}{3}\BBR{2+\f{\lambda}{\alpha_0(\delta)}}+\f{3}{2}|\gamma'-{\gamma_0}'|{\lambda}^{1/3}.
\end{aligned}
\end{equation}\\
Since $\delta$ is arbitrary and $\beta-\beta_0\in
L^{3p}(\delta,1)$, $\gamma'-\gamma_0' \in L^q(\delta,1)$, the
assumption (\ref{IDATAPROP2}) and last two inequalities imply that
for each $\delta\in (0,1)$
\begin{equation}
\alpha' \in L^{3p}(\delta,1) \bigcap L^q(\delta,1).
\end{equation}
This completes the proof.
\end{proof}\par


\section{Regularity}\label{secreg}

First, we claim that for each representative of the minimizer
$(\alpha,\beta,\gamma,v)\in \mathcal{A}_{\lambda}$ in the theorem
(\ref{ELWEAK}) we can alter $\alpha'$ on a set of measure zero
such that functions $G_1$ and $G_2$ defined in (\ref{G1DEF}) and
(\ref{G2DEF}) satisfy
\begin{equation*}
3\rho^{2/3} G_1(\rho)=\int_1^{\rho} \,s^{-1/3}G_2(s) +
\f{v(s)-v_0(s)}{h} \,ds + C_0, \;\; \tx{for all}\;\;\rho\in(0,1].
\end{equation*}
Indeed, let us fix representatives $(\alpha,\beta,\gamma,v)$ and
$(\alpha_0,\beta_0,\gamma_0,v_0)$. Define
\begin{equation}\label{ZFUNCDEF}
z(\rho)=\f{1}{3\rho^{2/3}}\int_1^{\rho} \,s^{-1/3}G_2(s) + \f{v(s)-v_0(s)}{h} \,ds +C_0\\
\end{equation}
and let $A=\{\rho\in(0,1]:\, G_1(\rho) \neq z(\rho)\}$. Take any
$\rho_0\in A$ and define
\begin{equation*}
y_0=\left.\BBR{z(\rho)-\varphi'( \beta\rho^{2/3}) -
2g'(3\gamma'\rho^{1/3})(\alpha_0/\rho)^{1/3}}\right|_{\rho=\rho_0}.
\end{equation*}
Then by (A1) and (A2) it follows that there exists a unique $x_0$
such that $h'(x_0)=y_0\BBR{{\rho_0}/{\alpha_0(\rho_0)}}^{2/3}$.
Now, by definition of $G_1$ we have for all $\rho\in(0,1]$
\begin{equation}\label{G1EXLP}
G_1(\rho)=\varphi'( \beta\rho^{2/3}) +
2g'(3\gamma'\rho^{1/3}/2)(\alpha_0/\rho)^{1/3} +
h'(\alpha')(\alpha_0/\rho)^{2/3}.
\end{equation}
Thus assigning $\alpha'(\rho_0)=x_0$ we get
$G_1(\rho_0)=z(\rho_0)$. In the end, after altering this way
$\alpha'$ on the set $A$, we get that $G_1(\rho)=z(\rho)$ for all
$\rho\in (0,1]$. Moreover by (\ref{G1ABSCONT}) we have $mA=0$ and
this finishes the proof.\par\medskip

The following regularity lemma requires a smoother initial iterate than before. In
particular we prove:

\begin{lmm}[\bf Regularity]\label{REGULARITY}
Let  $(\alpha,\beta,\gamma,v)\in \mathcal{A}_{\lambda}$ be the
minimizer of $I$ over $\mathcal{A}_{\lambda}$. Assume that the initial
iterate $(\alpha_0,\beta_0,\gamma_0,v_0)$ satisfies
(\ref{IDATAPROP}),
\begin{equation}\label{IDATAPROP3}
\alpha_0,\gamma_0\in C^1(0,1] \;\;\; \tx{and} \;\;\; \beta_0\in
C(0,1].
\end{equation}
Then
\begin{equation}
\alpha,\gamma,v\in C^1(0,1] \;\;\; \tx{and} \;\;\; \beta\in
C(0,1].
\end{equation}
\end{lmm}
\begin{proof}
Clearly, we can pick a representative $(\alpha,\beta,\gamma,v)$
such that $\alpha,\gamma,v\in C(0,1]$. Proceeding as in
(\ref{ALPHAPEST1}) and (\ref{ALPHAPEST2}), the constraints
$\f{\alpha-\alpha_0}{h}=3{\alpha_0}^{2/3}v$,
$\f{\gamma-\gamma_0}{h}=2{\alpha_0}^{1/3}v$ and
$\f{\beta-\beta_0}{h}=3v'$ imply for a.e. $\rho\in(0,1)$
\begin{equation}\label{BPVIAALPHAP}
\beta\rho^{2/3}=\alpha'(\rho/\alpha_0)^{2/3} + f_1(\rho)\\
\end{equation}
and
\begin{equation}\label{GPVIAALPHAP}
\f{3}{2}\gamma'\rho^{1/3}=\alpha' (\rho/\alpha_0)^{1/3} + f_2(\rho),\\
\end{equation}
where
\begin{equation*}
\begin{aligned}
 f_1(\rho)&=\beta_0\rho^{2/3} - \f{\alpha_0'\rho^{2/3}}{3{\alpha_0}^{2/3}}\BBR{1+\f{2\alpha}{\alpha_0}} \\
 f_2(\rho)&=\f{3}{2}{\gamma_0'}\rho^{1/3} -
 \f{\rho^{1/3}}{\alpha_0^{1/3}}\BBR{2+\f{\alpha}{\alpha_0}}.
\end{aligned}
\end{equation*}
We note that (\ref{IDATAPROP3}) implies that $f_1$ and $f_2$ are
continuous on $(0,1]$ functions.
\par\bigskip

First, we alter $\beta$ and $\gamma'$ so that equality in
(\ref{BPVIAALPHAP}) and (\ref{GPVIAALPHAP}) holds for all
$\rho\in(0,1)$. Hence by (\ref{G1EXLP}) we have for all $\rho\in
(0,1]$
\begin{equation}\label{G1EXPLS2}
\begin{aligned}
G_1(\rho)&=\varphi'\BBR{ \alpha'(\rho/\alpha_0)^{2/3} + f_2(\rho)} \\
&+2g'\BBR{\alpha'(\rho/\alpha_0)^{1/3}+f_1(\rho)}(\alpha_0/\rho)^{1/3}\\
&+h'(\alpha')(\alpha_0/\rho)^{2/3}.\\
\end{aligned}
\end{equation}
and this suggests to define $f:\RR_+ \times (0,1]\to \RR$ by
\begin{equation}\label{FDEF}
\begin{aligned}
f(x,\rho)&=\varphi'\BBR{ x(\rho/\alpha_0)^{2/3} + f_2(\rho)} \\
&+2g'\BBR{x(\rho/\alpha_0)^{1/3}+f_1(\rho)}(\alpha_0/\rho)^{1/3}\\
&+h'(x)(\alpha_0/\rho)^{2/3}.\\
\end{aligned}
\end{equation}\\
Now, define $A=\{\rho\in(0,1]:\, G_1(\rho)\neq z(\rho)\}$.
Clearly, $mA=0$ and note that from (\ref{G1EXPLS2}) it follows
\begin{equation}\label{G1REPR}
G_1(\rho)=f(\alpha',\rho)=z(\rho),\;\;\; \rho \notin A.
\end{equation}
Take $\rho_0\in A$. Then, as $\rho_0>0$ and $\alpha_0(\rho_0)>0$,
properties (A1)-(A3) imply that $f_x(x,\rho_0)>0$ for all $x\in
\RR_+$; moreover, $\lim_{x\to 0+}f(x,\rho_0)=-\infty$ and
$\lim_{x\to+\infty}f(x,\rho_0)=+\infty$. Hence there exists unique
$x_0\in \RR_+$ such that $f(x_0,\rho_0)=z(\rho_0)$.
\par\smallskip

At this point we are ready to assign new values for
$\alpha',\beta$ and $\gamma'$. Define
\begin{equation*}
\alpha'(\rho_0)=x_0,\;\;\;\beta(\rho_0)=\f{x_0}{{\alpha_0(\rho_0)}^{2/3}}+\f{f_1(\rho_0)}{\rho_0^{2/3}}\\
\end{equation*}
and
\begin{equation*}
\gamma'(\rho_0)=\f{2}{3}\BBR{
\f{x_0}{{\alpha_0(\rho_0)}^{1/3}}+\f{f_2(\rho_0)}{\rho_0^{1/3}} }.
\end{equation*}\\
This implies that (\ref{BPVIAALPHAP}) and (\ref{GPVIAALPHAP}) hold
at $\rho=\rho_0$ and hence by (\ref{G1EXLP})
\begin{equation}\label{G1REPR2}
G_1(\rho_0)=f(x_0,\rho_0)=f(\alpha'(\rho_0),\rho_0)=z(\rho_0).
\end{equation}
As  $\rho_0\in A$ was arbitrary (\ref{G1REPR}) and (\ref{G1REPR2})
imply
\begin{equation}\label{G1REPR3}
G_1(\rho)=f(\alpha',\rho)=z(\rho),\;\;\; \rho\in(0,1].
\end{equation}
Hence $G_1$ is continuous on $(0,1]$ and therefore $\alpha'>0$ for
all $\rho\in(0,1]$.\par\bigskip

Now, let us assume $\rho_k\to\rho_0$ and $\alpha'(\rho_k)\to
l\in[0,\infty]$ with $\rho_k,\rho_0\in(0,1],\;k\in \mathbb{N}$.
First, we claim that $l\in(0,\infty)$. Indeed, assume that $l=0$
or $l=+\infty$. Then by continuity of $\alpha_0$ we have
 $\alpha_0(\rho_k)\to\alpha_0(\rho_0)>0$ and hence properties
(A1)-(A3), together with continuity of $f_1$ and $f_2$, imply
$\lim_{k\to\infty} f(\alpha'(\rho_k),\rho_k)=\mp\infty$
respectively. Thus by continuity of $G_1$ and (\ref{G1REPR3}) we
have
\begin{equation}
G_1(\rho_0)=\lim_{k\to\infty}
G_1(\rho_k)=\lim_{k\to\infty}f(\alpha'(\rho_k),\rho_k)=\mp\infty
\end{equation}
which is a contradiction. Therefore we assume $l\in (0,\infty)$.
As $f_1$, $f_2$ are continuous on $(0,1]$, we
must have $\lim_{k\to\infty}
f(\alpha'(\rho_k),\rho_k)=f(l,\rho_0)$ and therefore by
(\ref{G1REPR3}) we get
\begin{equation}
\begin{aligned}
f(\alpha'(\rho_0),\rho_0)=G_1(\rho_0)&=\lim_{k\to\infty}
G_1(\rho_k)\\
&=\lim_{k\to\infty}f(\alpha'(\rho_k),\rho_k)=f(l,\rho_0).
\end{aligned}
\end{equation}
By the strict monotonicity of $f(\cdot,\rho_0)$ we get
$\alpha_0(\rho_0)=l$ and conclude that $\alpha'$ is
continuous on $(0,1]$.\par\smallskip

Finally, from the discussion above it follows that equalities
(\ref{BPVIAALPHAP}) and (\ref{GPVIAALPHAP}) hold for all
$\rho\in(0,1]$. The continuity of $f_1,f_2$ and
$\alpha'$ imply $\beta,\gamma'\in C(0,1].$ Moreover, as
$\f{\alpha-\alpha_0}{h}=3{\alpha_0}^{2/3}v$ for all
$\rho\in(0,1]$, we obtain $v\in C^1(0,1]$. This finishes the
proof.
\end{proof}



\section{Acknowledgements}
Research partially supported by the EU FP7-REGPOT project "Archimedes Center for
Modeling, Analysis and Computation" and the EU EST-project "Differential Equations and Applications
in Science and Engineering".

\end{document}